\documentclass[a4paper]{amsart}
\usepackage{amssymb}
\usepackage{graphicx}

\newtheorem{theorem}{Theorem}[section]
\newtheorem{lemma}[theorem]{Lemma}

\newtheorem{sublemma}[theorem]{Sublemma}
\newtheorem{cor}[theorem]{Corollary}
\newtheorem{proposition}[theorem]{Proposition}

\theoremstyle{definition}
\newtheorem{definition}[theorem]{Definition}
\newtheorem{example}[theorem]{Example}

\theoremstyle{remark}
\newtheorem{remark}[theorem]{Remark}

\numberwithin{equation}{section}
\newcommand{\Q}{\mathbb{Q}}
\newcommand{\C}{\mathbb{C}}

\newcommand{\R}{\mathbb{R}}

\DeclareMathOperator{\codim}{codim}

\DeclareMathOperator{\Aut}{Aut}
\DeclareMathOperator{\graph}{graph}
\DeclareMathOperator{\Iso}{Iso}

\title{Torus manifolds and non-negative curvature}

\author{Michael Wiemeler}
\address{Institut f\"ur Mathematik, Universit\"at Augsburg\\ D-86135 Augsburg\\ Germany}
\email{michael.wiemeler@math.uni-augsburg.de}
\thanks{}



\begin{document}
\begin{abstract}
  A torus manifold \(M\) is a \(2n\)-dimensional orientable manifold with an effective action of an \(n\)-dimensional torus such that \(M^T\neq \emptyset\).
  In this paper we discuss the classification of torus manifolds which admit an invariant metric of non-negative curvature.
  If \(M\) is a simply connected torus manifold which admits such a metric, then \(M\) is diffeomorphic to a quotient of a free linear torus action on a product of spheres.
  We also classify rationally elliptic torus manifolds \(M\) with \(H^{\text{odd}}(M;\mathbb{Z})=0\) up homeomorphism.
\end{abstract}

\maketitle


\section{Introduction}
\label{sec:intro}

The study of non-negatively curved manifolds has a long history in geometry.
In this note we discuss the classification of these manifolds in the context of torus manifolds.
A torus manifold \(M\) is a \(2n\)-dimensional closed orientable manifold with an effective action of an \(n\)-dimensional torus \(T\) such that \(M^T\neq \emptyset\).

Recently Spindeler \cite{spindeler} proved the Bott-conjecture for simply connected torus manifolds.
This conjecture implies that a non-negatively curved manifold is rationally elliptic.
Our first main result deals with rationally elliptic torus manifolds:

\begin{theorem}[{Theorem \ref{sec:torus-manifolds-with}}]
\label{sec:introduction-1}
  Let \(M\) be a simply connected rationally elliptic torus manifold
  with \(H^{\text{odd}}(M;\mathbb{Z})=0\). 
    Then \(M\)  is homeomorphic to a quotient of a free linear torus action on a product of spheres.
\end{theorem}

Our second main theorem is as follows.

\begin{theorem}[{Theorem \ref{sec:proof-theor-refs-1}}]
\label{sec:introduction}
  Let \(M\) be a simply connected non-negatively curved torus manifold.
  Then \(M\) is equivariantly diffeomorphic to a quotient of a free linear torus action on a product of spheres.
\end{theorem}

In the situation of the theorem the torus action on the quotient \(N/T'\) of a product of spheres \(N=\prod_{i<r} S^{2n_i}\times \prod_{i\geq r} S^{2n_i+1}\)  by a free linear action  of a torus \(T'\) is defined as follows.
Let \(T\) be a maximal torus of \(\prod_{i<r} SO(2n_i+1)\times \prod_{i\geq r} SO(2n_i+2)\).
Then there is a natural linear action of \(T\) on \(N\).
Moreover, \(T'\) can be identified with a subtorus of \(T\).
Therefore \(T/T'\) acts on \(N/T'\).
If the dimension of \(T'\) is equal to the number of odd-dimensional factors in the product \(N\), then \(N/T'\) together with the action of \(T/T'\) is a torus manifold.

We also show that the fundamental group of a non-simply connected
torus manifold of dimension \(2n\) with an invariant metric of
non-negative curvature is isomorphic to \(\mathbb{Z}^{k}_2\) with
\(k\leq n-1\). In particular, every such manifold is finitely covered by a manifold as in
Theorem~\ref{sec:introduction}.

Theorems \ref{sec:introduction-1} and \ref{sec:introduction} are already known in dimension four.
As is well known a simply connected rationally elliptic four-manifold is homeomorphic to \(S^4\), \(\C P^2\), \(S^2\times S^2\), \(\C P^2\# \C P^2\) or \(\C P^2\# \overline{\C P^2}\).
Furthermore, a simply connected non-negatively curved four-dimensional torus manifold is diffeomorphic to one of the manifolds in the above list (see \cite{MR2784821}, \cite{MR1272983}, \cite{kleiner90:_rieman}, \cite{grove13}).
Moreover, by \cite{MR1104531} or \cite{galaz-garciaar:_cohom}, the \(T^2\)-actions on these spaces are always equivalent to a torus action as described above.

We should note here that it has been shown by Grove and Searle \cite{MR1255926} that a simply connected torus manifold which admits an invariant metric of positive curvature is diffeomorphic to a sphere or a complex projective space.
Moreover, it has been shown by Gurvich in his thesis \cite{gurvich08:_some_resul_topol_of_quasit} that the orbit space of a rationally elliptic quasitoric manifold is face-preserving homeomorphic to a product of simplices.
 This condition on the orbit space is satisfied if and only if the quasitoric manifold is a quotient of a free torus action on a product of odd-dimensional spheres. 

In dimension six Theorem~\ref{sec:introduction-1} follows from these
results of Gurvich, the classification of simply connected six
dimensional torus manifolds with \(H^3(M;\mathbb{Z})=0\) given by
Kuroki \cite{kuroki13:_two} and a characterization of the cohomology rings of simply
connected six-dimensional rationally elliptic manifolds given by
Herrmann \cite{herrmann14:_ration_ellip_kruem_kohom}.

In \cite{MR1255926} \((2n+1)\)-dimensional positively curved manifolds
with isometric actions of an \((n+1)\)-dimensional manifolds were also
classified.
We expect that a result similar to Theorem~\ref{sec:introduction}
holds for isometric actions of \(n+1)\)-dimensional tori on
non-negatively curved \(2n+1\)-dimensional manifolds with
one-dimensional orbits.
We will discuss the details of this in a subsequent paper.

We apply Theorem \ref{sec:introduction-1} to rigidity problems in toric topology.
As a consequence we get the following theorem:

\begin{theorem}[Corollary \ref{sec:appl-rigid-probl-1}]
\label{sec:introduction-2}
  Let \(M\) be a simply connected torus manifold with \(H^*(M;\mathbb{Z})\cong H^*(\prod_i \C P^{n_i};\mathbb{Z})\).
  Then \(M\) is homeomorphic to \(\prod_i \C P^{n_i}\).
\end{theorem}

This theorem is a stronger version of a result of Petrie \cite{MR0322893} related to his conjecture on circle actions on homotopy complex projective spaces.
This conjecture states that if \(f:M\rightarrow \C P^n\) is a homotopy equivalence and \(S^1\) acts non-trivially on the manifold \(M\), then \(f^*(p(\C P^n))=p(M)\), where \(p(M)\) denotes the total Pontrjagin class of \(M\).
Petrie showed that his conjecture holds if there is an action of an \(n\)-dimensional torus on \(M\).
Moreover, the conjecture has been shown by Dessai and Wilking \cite{MR2114425} for the case that there is an effective action of a torus of dimension greater than \(\frac{n+1}{4}\) on \(M\).
For more results related to this conjecture see the references in \cite{MR2114425}.

The proof of Theorem~\ref{sec:introduction-1} consists of two
steps.
In a first step we show that the homeomorphism type of a simply
connected torus
manifold \(M\), whose cohomology with integer coefficients vanishes in
odd degrees, depends only on
the isomorphism type of the face poset of \(M/T\) and the
characteristic function of \(M\).
Then all possible face posets of \(M/T\) under the condition that \(M\) is
rationally elliptic  are determined.
By the first step all such manifolds are homeomorphic to quotients of
the moment angle complex associated to these posets.
As it turns out these moment angle complexes are products of spheres.

The proof of Theorem~\ref{sec:introduction} is based on the
computations of the face posets from above.
With Spindeler's results from
\cite{spindeler} one can see that all faces of \(M/T\) are
diffeomorphic after smoothing the corners to standard discs.
When this is established, generalizations of results from \cite{MR3030690}
imply the theorem.

This paper is organized as follows.
In Section \ref{sec:pre} we discuss results of Masuda and Panov about torus manifolds with vanishing odd degree cohomology.
In Section \ref{sec:simplyfy} we introduce a construction which simplifies the torus action on  a torus manifold.
In Sections \ref{sec:torus-with-non-neg}, \ref{sec:rigidity} and
\ref{sec:towards} we prove Theorems \ref{sec:introduction-1},
\ref{sec:introduction-2} and \ref{sec:introduction}, respectively.
In the last Section~\ref{sec:non-simply-con} we discuss non-simply
connected non-negatively curved torus manifolds.

I would like to thank Wolfgang Spindeler for sharing his results from \cite{spindeler}.
I would also like to thank Fernando Galaz-Garcia, Martin Kerin,
Marco Radeschi and Wilderich Tuschmann for comments on earlier
versions of this paper.
I would also like to thank the anonymous referee for suggestions which
helped to improve the exposition of the article.

\section{Preliminaries}
\label{sec:pre}

Before we prove our results, we review some results of Masuda and Panov \cite{MR2283418} about torus manifolds with vanishing odd degree cohomology.

They have shown that a smooth torus manifold \(M\) with \(H^{\text{odd}}(M;\mathbb{Z})=0\) is locally standard.
This means that each point in \(M\) has an invariant neighborhood which is weakly equivariantly homeomorphic to an open invariant subset of the standard \(T^n\)-representation on \(\C^n\).
Moreover, the orbit space is a nice manifold with corners such that all faces of \(M/T\) are acyclic \cite[Theorem 9.3]{MR2283418}.
Here a manifold with corners is called nice if each of its codimension-\(k\) faces is contained  in exactly \(k\) codimension-one faces.
A codimension-one face of \(M/T\) is also called a facet of \(M/T\).
Moreover, following Masuda and Panov, we count \(M/T\) itself as a codimension-zero face of \(M/T\).

 The faces of \(M/T\) do not have to be contractible.
 But we show in Section~\ref{sec:simplyfy} that the action on \(M\) can be changed in such a way that all faces become contractible without changing the face-poset of \(M/T\).
 This new action might be non-smooth (see Remark \ref{sec:simpl-torus-acti-1}).
 But it always admits a canonical model  over a topological nice manifold with corners as described below.

For a facet \(F\) of \(Q=M/T\) denote by \(\lambda(F)\) the isotropy group of a generic point in \(\pi^{-1}(F)\), where \(\pi:M\rightarrow M/T\) is the orbit map.
Then \(\lambda(F)\) is a circle subgroup of \(T\).
Let
\begin{equation*}
  M_Q(\lambda)=Q\times T/\sim,
\end{equation*}
where two points \((x_i,t_i)\in Q\times T\), \(i=1,2\), are identified if and only if \(x_1=x_2\) and \(t_1t_2^{-1}\) is contained in the subtorus of \(T\) which is generated by the \(\lambda(F)\) with \(x_1\in F\).
There is a \(T\)-action on \(M_Q(\lambda)\), induced by multiplication on the second factor in \(Q\times T\).
Then, by \cite[Lemma 4.5]{MR2283418}, there is an equivariant homeomorphism
\begin{equation*}
  M_Q(\lambda)\rightarrow M.
\end{equation*}

For every map \(\lambda:\{\text{facets of }Q\}\rightarrow \{\text{one-dimensional subtori of }T\}\) such that \(T\) is isomorphic to \(\lambda(F_1)\times\dots\times \lambda(F_n)\), whenever the intersection of \(F_1\cap\dots\cap F_n\) is non-empty, the model \(M_Q(\lambda)\) is a manifold.

The canonical model is equivariantly homeomorphic to a quotient of a free torus action on the moment angle complex \(Z_Q\)  associated to \(Q\).
Here \(Z_Q\) is defined as follows:
\begin{equation*}
  Z_Q= Q\times T_Q/\sim.
\end{equation*}

Here \(T_Q\) is the torus \(S^1_{1}\times\dots\times S^1_{k}\), where \(k\) is the number of facets of \(Q\).
The equivalence relation \(\sim\) is defined as follows.
Two points \((q_i,t_i)\in Q\times T_{Q}\) are identified if \(q_1=q_2\) and \(t_1t_2^{-1}\in \prod_{i\in S(q_1)} S_i^1\), where \(S(q_1)\) is the set of those facets of \(Q\) which contain \(q_1\).

The torus which acts freely on \(Z_Q\) with quotient \(M_Q(\lambda)\) is given by the kernel of a homomorphism \(\psi:T_Q\rightarrow T\), such that the restriction of \(\psi\) to \(S^1_i\) induces an isomorphism \(S_i^1\rightarrow \lambda(F_i)\). 

\begin{example}
  If \(Q=\Delta^n\) is an \(n\)-dimensional simplex, then \(T_Q\) is an \((n+1)\)-dimensional torus.
  Moreover, \(Z_Q\) is equivariantly homeomorphic to \(S^{2n+1}\subset \C^{n+1}\) with the standard linear torus action.
\end{example}

\begin{example}
   If \(Q=\Sigma^n\) is the orbit space of the standard linear torus action on \(S^{2n}\), then \(T_Q\) is \(n\)-dimensional.
  Moreover, \(Z_Q\) is equivariantly homeomorphic to \(S^{2n}\subset \C^{n}\oplus \R\) with the standard linear torus action.
\end{example}

\begin{example}
  Let \(Q_1\) and \(Q_2\) be two nice manifolds with corners.
  If \(Q=Q_1\times Q_2\), then \(T_Q\cong T_{Q_1}\times T_{Q_2}\) and \(Z_Q\) is equivariantly homeomorphic to \(Z_{Q_1}\times Z_{Q_2}\).
\end{example}

Now assume that \(M\) is a torus manifold with \(H^{\text{odd}}(M;\Q)=0\).
This condition is always satisfied if \(M\) is rationally elliptic because \(\chi(M)=\chi(M^T)>0\).
  Then the torus action on \(M\) might not be locally standard and \(M/T\) might not be a manifold with corners.
  But \(M/T\) still has a  face-structure induced by its stratification by connected orbit types.
  It is defined as in \cite{MR0375357}.
  A \(k\)-dimensional face is a component \(C\) of \(M^{T^{n-k}}/T\) such that the identity component of the isotropy group of a generic point in \(C\) is equal to \(T^{n-k}\), where \(T^{n-k}\) is a subtorus of codimension \(k\) in \(T\).
  The faces of \(M/T\) defined in this way have the following properties:
  \begin{itemize}
  \item It follows from localization in equivariant cohomology that
    the cohomology of every \(M^{T^{n-k}}\) is concentrated in even degrees.
    Therefore every component of \(M^{T^{n-k}}\) contains a \(T\)-fixed point.
    This is equivalent to saying that
 each face of \(M/T\) contains at least one vertex, i.e. a face of dimension zero.
  \item By an investigation of the local weights of the action, one sees that each face of \(M/T\) of codimension \(k\) is contained in exactly \(k\) faces of codimension \(1\).
  \item   The vertex-edge-graph of each face is connected. (see \cite[Proposition 2.5]{MR0375357})
  \end{itemize}

\section{Simplifying torus actions}
\label{sec:simplyfy}

In this section we describe an operation on locally standard torus manifolds \(M\) which simplifies the torus action on \(M\).
For this construction we need the following two lemmas.

\begin{lemma}
\label{sec:locally-stand-torus}
  Let \(M\) be a topological \(n\)-manifold with \(H^*(M;\mathbb{Z})\cong H^*(S^n;\mathbb{Z})\). 
  Then there is a contractible compact \((n+1)\)-manifold \(X\) such that \(\partial X=M\).
  Moreover, \(X\) is unique up to homeomorphism relative \(M\).
  In particular, every homeomorphism of \(M\) extends to a homeomorphism of \(X\).
\end{lemma}
\begin{proof}
  For \(n\leq 2\), this follows from the classification of manifolds of dimension \(n\).
  If \(n=3\) then this follows from the proof of Corollary 9.3C and Corollary 11.1C of \cite{MR1201584}.
  For \(n\geq 4\) this follows from the proof of \cite[Corollary 11.1]{MR1201584}.
\end{proof}

\begin{lemma}
\label{sec:locally-stand-torus-2}
  Let \(Q_1,Q_2\) be two nice manifolds with corners   of the same dimension such that all faces of \(Q_i\), \(i=1,2\), are contractible.
  If there is an isomorphism of their face-posets \(\phi\colon \mathcal{P}(Q_1)\rightarrow \mathcal{P}(Q_2)\), then there is a face-preserving homeomorphism \(f\colon Q_1\rightarrow Q_2\), such that, for each face \(F\)  of \(Q_1\), \(f(F)=\phi(F)\).
\end{lemma}
\begin{proof}
  We construct \(f\) by induction on the \(n\)-skeleton of \(Q_1\).
  There is no problem to define \(f\) on the \(0\)-skeleton.
  Therefore assume that \(f\) is already defined on the \((n-1)\)-skeleton.
 
  Let \(F\) be a \(n\)-dimensional face of \(Q_1\).
  Then \(f\) restricts to an homeomorphism \(\partial F\rightarrow \partial \phi(F)\).
  Because \(F\) and \(\phi(F)\) are contractible manifolds with boundary \(\partial F\), \(f\) extends to a homeomorphism \(F\rightarrow \phi(F)\).
  This completes the proof.
\end{proof}

Now let \(Q\) be a nice manifold with corners and \(F\) a face of \(Q\) of positive codimension which is a homology disc.
Let \(X\) be a homology disc with \(\partial X = {\partial F}\).
Then \(X\cup_{\partial F}F\) is a homology sphere and therefore bounds a contractible manifold \(Y\).
We equip \(Y\) with a face structure such that the facets of \(Y\) are given by \(F\) and \(X\) and the lower dimensional faces coincide with the faces of \(F\) in \(\partial F=F\cap X\).
With this face-structure \(X\) and \(F\) become nice manifolds with corners.

Let \(k=\dim Q-\dim F\).
Then define
\begin{equation*}
  S_{X,F}=Y\times\Delta^{k-1}\cup_{X\times \Delta^{k-1}} X\times \Delta^{k}.
\end{equation*}
Then \(F'=F\times \Delta^{k-1}\) is a facet of \(S_{X,F}\) and we define
\begin{equation*}
  \alpha_{X,F}(Q)= Q-(F\times \Delta^k)\cup_{F'} S_{X,F}.
\end{equation*}

Then \(\alpha_{X,F}(Q)\) is naturally a nice manifolds with corners.

\begin{lemma}
\label{sec:simpl-torus-acti}
  Let \(Q\), \(F\), \(X\) as above.
  If \(dim F\geq 3\) and all faces of \(Q\) of dimension greater than \(\dim F\) are contractible, then \(Q'=\alpha_{F,X}(\alpha_{X,F}(Q))\) and \(Q\) are face-preserving homeomorphic.
\end{lemma}
\begin{proof}
  It is clear from the construction above that \(\mathcal{P}(Q')\) and \(\mathcal{P}(Q)\) are isomorphic.
  Moreover, the \(\dim F\)-skeleta of \(Q\) and \(Q'\) are face-preserving homeomorphic.
  Therefore, by using Lemma~\ref{sec:locally-stand-torus}, the statement follows by an induction as in the proof of Lemma~\ref{sec:locally-stand-torus-2}.
\end{proof}

If \(M\) is a locally standard torus manifold over \(Q\), then we can construct a torus manifold \(M'\) with orbit space \(\alpha_{X,F}(Q)\) as follows.
Let \(T^k\) act on \(Y'=\partial(Y\times D^{2k})\) (without the face-structure) by the standard action on the second factor.
Choose an isomorphism \(T\cong T^k\times T^{n-k}\) which maps \(\lambda(F)\) to the first factor \(T^k\).
If \(\dot{F}\) denotes \(F\) with a small collar of its boundary removed, then a small neighborhood of \(\pi^{-1}(\dot{F})\) is equivariantly homeomorphic to 
\(\dot{F}\times T^{n-k}\times D^{2k}\).

We define
\begin{equation*}
  \beta_{X,F}(M)=M-(\pi^{-1}(\dot{F})\times D^{2k})\cup_{\partial (\dot{F}\times D^{2k} \times T^{n-k})} (Y'-\dot{F}\times D^{2k})\times T^{n-k}.
\end{equation*}

Since \(Y'\) is the simply connected boundary of a contractible manifold, it is homeomorphic to a sphere.
Moreover, if \(F\) is contractible, then \(\dot{F}\times D^{2k}\) is contractible with simply connected boundary.
Hence, it follows from Schoenflies' Theorem that \(\dot{F}\times D^{2k}\) is a disc and that we may assume that \(F\times D^{2k}\) is embedded in \(Y'\) as the upper hemisphere.
Therefore it follows that \(\beta_{X,F}(M)\) is homeomorphic to \(M\) if \(F\) is contractible.

\begin{theorem}
\label{sec:simpl-torus-acti-3}
  Let \(M\) be a simply connected torus manifold with \(H^{\text{odd}}(M;\mathbb{Z})=0\).
  Then \(M\) is determined by \((\mathcal{P}(M/T),\lambda)\) up to homeomorphism.
\end{theorem}
\begin{proof}
  The first step is to simplify the action on \(M\) in such a way that all faces \(F\) of \(M/T\) become contractible.
  Then the statement will follow from Lemma~\ref{sec:locally-stand-torus-2} and \cite[Lemma 4.5]{MR2283418}.
  
  At first assume that \(\dim M\leq 6\).
  Then because \(\pi_1(M)=0\) all faces of \(M/T\) are contractible.
  Therefore the theorem follows in this case.

  Next assume that \(\dim M\geq 8\).
  Because \(M\) is simply connected, \(M/T\) is contractible.
  We simplify the torus action on \(M\) by using the operations \(\beta_{X,F}\) applied by a downwards induction on the dimension of \(F\) beginning with the codimension one faces.
At first assume that \(\dim F\geq 3\) and that all faces of dimension greater than \(\dim F\) are already contractible.
Let \(X\) be a contractible manifold with boundary \(\partial F\).
Then \(\beta_{F,X}(\beta_{X,F}(M))\) is homeomorphic to \(\beta_{X,F}(M)\).
But by Lemma \ref{sec:simpl-torus-acti}, \(\alpha_{F,X}(\alpha_{X,F}(M/T))\) is face-preserving homeomorphic to \(M/T\).
Therefore it follows that \(\beta_{F,X}(\beta_{X,F}(M))\) is homeomorphic to \(M\).
Hence, \(\beta_{X,F}(M)\) is also homeomorphic to \(M\).

If \(\dim F\leq 2\), it follows from the classification of two- and one-dimensional manifolds that there is nothing to do.
\end{proof}

\begin{remark}
\label{sec:simpl-torus-acti-1}
  Since not every three-dimensional homology sphere bounds a contractible smooth manifold, the torus action on \(\beta_{X,F}(M)\) might be non-smooth.
  Therefore with our methods we cannot prove that the diffeomorphism type of \(M\) is determined by \((\mathcal{P}(M/T),\lambda)\).
\end{remark}

Using results of Masuda and Panov \cite{MR2283418} and the author's methods from \cite{MR2885534} together with Theorem~\ref{sec:simpl-torus-acti-3}, one can prove the following partial generalization of Theorem 2.2 of \cite{MR2885534}.

\begin{theorem}
\label{sec:simpl-torus-acti-2}
  Let \(M\) and \(M'\) be simply connected torus manifolds of dimension \(2n\) with \(H^*(M;\mathbb{Z})\) and \(H^*(M';\mathbb{Z})\) generated in degree \(2\).
Let \(m,m'\) be the numbers of characteristic submanifolds of \(M\) and \(M'\), respectively.
Assume that \(m\leq m'\).
  Furthermore, let \(u_1,\dots,u_m\in H^2(M)\) be the Poincar\'e-duals of the characteristic submanifolds of \(M\) and  \(u_1',\dots,u_{m'}'\in H^2(M')\) the Poincar\'e-duals of the characteristic submanifolds of \(M'\).
  If there is a ring isomorphism \(f:H^*(M)\rightarrow H^*(M')\) and a permutation \(\sigma:\{1,\dots,m'\}\rightarrow \{1,\dots,m'\}\) with \(f(u_i)=\pm u_{\sigma(i)}'\), for \(i=1,\dots,m\), then \(M\) and \(M'\) are homeomorphic.
\end{theorem}

Here a characteristic submanifold of a torus manifold is a codimension-two-submanifold which is fixed by some circle-subgroup of the torus and contains a \(T\)-fixed point. Each characteristic submanifold is the preimage of some facet of \(M/T\) under the orbit map.

In \cite{MR0322893} Petrie proved that if an \(n\)-dimensional torus acts on a homotopy complex projective space \(M\) of real dimension \(2n\), then the Pontrjagin classes of \(M\) are standard.
In fact a much stronger statement holds. 

\begin{cor}
  Let \(M\) be a torus manifold which is homotopy equivalent to \(\C P^n\).
  Then \(M\) is homeomorphic to \(\C P^n\).
\end{cor}
\begin{proof}
  By Corollary 7.8 of \cite{MR2283418}, the cohomology ring of \(M\) can be computed from the face-poset of the orbit space.
  In particular all the Poincar\'e duals of the characteristic submanifolds of \(M\) are generators of \(H^2(M;\mathbb{Z})\).
  Now it follows from Theorem~\ref{sec:simpl-torus-acti-2} that \(M\) is homeomorphic to \(\C P^n\).
\end{proof}

\section{Rationally elliptic torus manifolds}
\label{sec:torus-with-non-neg}

In this section we prove the following theorem.

\begin{theorem}
\label{sec:torus-manifolds-with}
  Let \(M\) be a simply connected rationally elliptic torus manifold
  with \(H^{\text{odd}}(M;\mathbb{Z})=0\).
  Then \(M\) is homeomorphic to a quotient of a free linear torus action on a product of spheres. 
\end{theorem}

Since the proof of this theorem is very long we give a short outline
of its proof.

\begin{proof}[Sketch of proof]
  In a first step (Lemma~\ref{sec:torus-manifolds-with-1}) we will
  show that each two-dimensional face of the orbit space of \(M\)
  contains at most four vertices.
  Then we will show in Proposition~\ref{sec:torus-manifolds-with-2}
  that a nice manifold with corners \(Q\), whose two-dimensional faces
  contain at most four vertices and all of whose faces are acyclic, is
  combinatorially equivalent to a product \(\prod_{i<r} \Sigma^{n_i}\times
  \prod_{i\geq r} \Delta^{n_i}\). Here \(\Sigma^{n}\) is the orbit space of
  the linear \(T^n\)-action on \(S^{2n}\) and \(\Delta^{n}\) is the
  \(n\)-dimensional simplex.

  When this is achieved Theorem~\ref{sec:torus-manifolds-with} will
  follow from Theorem~\ref{sec:simpl-torus-acti-3} and the structure
  results described in Section~\ref{sec:pre}.

  The proof of Proposition~\ref{sec:torus-manifolds-with-2} is by
  induction on the dimension of \(Q\).
  The case \(\dim Q=2\) is obvious.
  Therefore we may assume that \(n=\dim Q>2\).
  
  We consider a facet \(F\) of \(Q\). By the induction hypothesis we
  know that there is a combinatorial equivalence
  \begin{equation}
\label{eq:5}
    F\cong \prod_{i<r}\Sigma^{n_i}\times \prod_{i\geq r}\Delta^{n_i}.
  \end{equation}
  The facets \(G_k\) of \(Q\) which meet \(F\) intersect \(F\) in a
  disjoint union of facets of \(F\).
  Since the facets of \(F\) are all of the form
  \begin{equation*}
    \tilde{F}\times\prod_{i_0\neq i<r}\Sigma^{n_i}\times
    \prod_{i_0\neq i\geq r}\Delta^{n_i},
  \end{equation*}
  where \(\tilde{F}\) is a facet of the \(i_0\)-th factor in the
  product (\ref{eq:5}).

  Hence, it follows that each \(G_k\) ``belongs'' to a factor
  \(\Gamma_{j(k)}\) of the product (\ref{eq:5}).
  There are seven cases of how the \(G_k\) which belong to the same
  factor can intersect.
  We call four of these cases exceptional (these are the cases
  \ref{item:3}, \ref{item:4}, \ref{item:6} and \ref{item:8} in the
  list in the proof of Proposition \ref{sec:torus-manifolds-with-2})
  and three of them generic (these are the cases \ref{item:5},
  \ref{item:7}, \ref{item:9}).

  Depending on which cases occur we determine the combinatorial type
  of \(Q\) in Lemmas \ref{sec:torus-manifolds-with-3},
  \ref{sec:rati-ellipt-torus}, \ref{sec:rati-ellipt-torus-1} and
  \ref{sec:torus-manifolds-with-4}.
  With these Lemmas we complete the proof of
  Proposition~\ref{sec:torus-manifolds-with-2}.

  The proofs of the above lemmas are again subdivided into several
  sublemmas.
  In Sublemmas \ref{sec:rati-ellipt-torus-5},
  \ref{sec:rati-ellipt-torus-4}, \ref{sec:rati-ellipt-torus-7} and
  \ref{sec:rati-ellipt-torus-6} we determine the combinatorial type of
  those facets of \(Q\) which belong to a factor \(\Gamma_{j_0}\)
  where one of the exceptional cases occurs.

  Then in Sublemma~\ref{sec:rati-ellipt-torus-11} we use this
  information and Lemma~\ref{sec:rati-ellipt-torus-3} to show that
  \(F\) has at most one factor where one of the exceptional cases can
  appear.
  In Lemma~\ref{sec:rati-ellipt-torus-3} we show that a certain poset
  is not the poset of a nice manifold with corners with only 
  acyclic faces.

  Assuming that one of the exceptional cases appears at the factor
  \(\Gamma_{j_0}\) of \(F\) the combinatorial types of the \(G_k\)
  which do not belong to \(\Gamma_{j_0}\) are then determined in
  Sublemma~\ref{sec:rati-ellipt-torus-8}.
  With the information gained in this sublemma together with the
  results from Sublemmas \ref{sec:rati-ellipt-torus-5},
  \ref{sec:rati-ellipt-torus-4}, \ref{sec:rati-ellipt-torus-7} and
  \ref{sec:rati-ellipt-torus-6} we then can determine the
  combinatorial type of \(Q\) for the case that at one factor of \(F\)
  one of the exceptional cases occurs.
  This completes the proofs of Lemmas \ref{sec:torus-manifolds-with-3},
  \ref{sec:rati-ellipt-torus} and \ref{sec:rati-ellipt-torus-1}.

  So we are left with the case where no exceptional case appears at
  the factors of \(F\).
  For this case we determine in
  Sublemma~\ref{sec:rati-ellipt-torus-10} the combinatorial types of
  the \(G_k\).
  Moreover, in Sublemma~\ref{sec:rati-ellipt-torus-9} we show that
  there is exactly one facet \(H\) of \(Q\) which does not meet \(F\).
  We determine the combinatorial type of \(H\) in the same sublemma.
  With the information on the combinatorial types of all facets of
  \(Q\) we then can determine the combinatorial type of \(Q\).
  This completes the proof of Lemma~\ref{sec:torus-manifolds-with-4}.
\end{proof}

For the proof of Theorem~\ref{sec:torus-manifolds-with} we need the following lemmas.

\begin{lemma}
\label{sec:torus-manifolds-with-1}
  Let \(M\) be a torus manifold with \(H^{\text{odd}}(M;\mathbb{Q})=0\). 
 Assume that \(M\) admits an invariant metric of non-negative sectional curvature or \(M\) is rationally elliptic.
 Then each two-dimensional face of \(M/T\) contains at most four vertices.
\end{lemma}
\begin{proof}
  At first note that, by localization in equivariant cohomology, the odd-degree cohomology of all fixed point components of all subtori of \(T\) vanishes.
 
  A two-dimensional face \(F\) is the image of a fixed point component \(M_1\) of a codimension-two subtorus of \(T\) under the orbit map.
  Therefore it follows from the classification of four-dimensional \(T^2\)-manifolds given in \cite{MR0268911} and \cite{MR0348779} that the orbit space is homeomorphic to a two-dimensional disk.
  If \(M\) admits an invariant metric of non-negative sectional curvature, then the same holds for \(M_1\).
  Therefore it follows from the argument in the proof of \cite[Lemma 4.1]{MR2784821} that there are at most four vertices in \(F\).

  Now assume that \(M\) is rationally elliptic.
  Then, by \cite[Corollary 3.3.11]{MR1236839}, the minimal model \(\mathcal{M}(M_1)\) of \(M_1\) is elliptic.
  The number of vertices in \(F\) is equal to the number of fixed points in \(M_1\).
  Since \(\chi(M_1^{T})=\chi(M_1)\), it is also equal to the Euler-characteristic of \(M_1\).
  By \cite[Theorem 32.6]{MR1802847} and \cite[Theorem 32.10]{MR1802847}, we have
  \begin{equation*}
    4\geq 2\dim\Pi_\psi^2(M_1)=2b_2(M_1).
  \end{equation*}
  Here \(\Pi_{\psi}^2(M_1)\) denotes the second pseudo-dual rational homotopy group of \(M_1\).
  Therefore, \(\chi(M_1)\leq 4\) and there are at most four vertices in \(F\).
\end{proof}

\begin{remark}
  If \(M\) is a rationally elliptic torus manifold, then we always
  have \(H^{\text{odd}}(M;\mathbb{Q})=0\) since
  \(\chi(M)=\chi(M^T)>0\).
  Therefore \(H^{\text{odd}}(M;\mathbb{Z})=0\) if and only if
  \(H^*(M;\mathbb{Z})\) is torsion-free.
\end{remark}

\begin{lemma}
  \label{sec:rati-ellipt-torus-3}
  For \(n>2\), there is no  \(n\)-dimensional nice manifold with corners whose faces are all acyclic,
  such that each facet is combinatorially equivalent to an
  \((n-1)\)-dimensional cube and the intersection of any two facets has
  two components.
\end{lemma}
\begin{proof}
  Assume that there is such a manifold \(Q\) with corners.
  Then the boundary of \(Q\) is a homology sphere.
  Moreover by applying the construction \(\alpha\) from
  Section~\ref{sec:simplyfy} to \(Q\) we may assume that all faces of
  \(Q\) of codimension at least one are contractible.

  Let \(Q'=[-1,1]^n/\mathbb{Z}_2\), where \(\mathbb{Z}_2\) acts on
  \([-1,1]^n\) by multiplication with \(-1\) on each factor.
  Then the boundary of \(Q'\) is a real projective space of dimension
  \(n-1\).

  Moreover, for each facet \(F_i\) of \(Q\) and each facet \(F'_i\) of
  \(Q'\) there is an isomorphism of face posets
  \(\mathcal{P}(F_i)\rightarrow \mathcal{P}(F'_i)\)
  such that \(F_i\cap F_j\) is mapped to \(F_i'\cap F_j'\).

  Since there are automorphisms of \(\mathcal{P}(F_i)\) which
  interchange the two components of \(F_i\cap F_j\) and leave the
  other facets of \(F_i\) unchanged, we can glue these isomorphisms
  together to get an isomorphism of face posets
  \begin{equation*}
    \mathcal{P}(Q)\rightarrow \mathcal{P}(Q').
  \end{equation*}
Since the faces of \(Q\) and \(Q'\) of codimension at least one are
contractible we obtain a homeomorphism \(\partial Q\rightarrow \partial
Q'\).
This is a contradiction because \(\partial Q\) is a homology sphere and
\(\partial Q'\) a projective space.
\end{proof}

\begin{proposition}
\label{sec:torus-manifolds-with-2}
  Let \(Q\) be a nice manifold with only acyclic faces such that each
  two-dimensional face of \(Q\) has at most four vertices.
  Then \(\mathcal{P}(Q)\) is isomorphic to the face poset of a product \(\prod_i \Sigma^{n_i} \times \prod_i \Delta^{n_i}\).
  Here \(\Sigma^m\) is the orbit space of the linear \(T^m\)-action on \(S^{2m}\) and \(\Delta^m\) is an \(m\)-dimensional simplex.
\end{proposition}
\begin{proof}
  We prove this proposition by induction on the dimension of \(Q\).
  If \(\dim Q=2\), there is nothing to show.

  Therefore let us assume that \(\dim Q>2\) and that all facets of \(Q\) are combinatorially equivalent to  a product of \(\Sigma^{n_i}\)'s and \(\Delta^{n_i}\)'s.

  Let \(F\) be a facet of \(Q\), such that \(F\) is combinatorially equivalent to \(\prod_i \Gamma_i\), where \(\Gamma_i= \Sigma^{n_i}\) for \(i<r\) and \(\Gamma_i=\Delta^{n_i}\) for \(i\geq r\).
  We fix this facet \(F\) for the rest of this section.

  In the following we will denote \(\tilde{\Gamma}_i=\Sigma^{n_i-1}\) and \(\bar{\Gamma}_i=\Sigma^{n_i+1}\) if \(i<r\)  or \(\tilde{\Gamma}_i=\Delta^{n_i-1}\) and \(\bar{\Gamma}_i=\Delta^{n_i+1}\) if \(i \geq r\).

Each facet of \(Q\) which meets \(F\) intersects \(F\) in a union of facets of \(F\).
Since \(\mathcal{P}(F)\cong \mathcal{P}(\prod_i \Gamma_i)\), the facets of \(F\) are of the form
\begin{equation*}
  F_j\times \prod_{i\neq j} \Gamma_i,
\end{equation*}
where \(F_j\) is a facet of \(\Gamma_j\).
Therefore each facet  \(G_{k}\) of \(Q\) which meets \(F\) ``belongs'' to a factor \(\Gamma_{j(k)}\) of \(F\), i.e.
  \begin{equation*}
    F\cap G_{k}\cong\prod_{i\neq j(k)} \Gamma_i \times \tilde{F}_{k},
  \end{equation*}
  where \(\tilde{F}_{k}\) is a union of facets of the \(j(k)\)-th factor \(\Gamma_{j(k)}\) in \(F\).

If \(\dim \Gamma_j=1\), then \(\Gamma_j\) is combinatorially equivalent to an interval.
Hence, it has two facets which do not intersect.
Therefore in this case there are at most two facets of \(Q\) which belong to \(\Gamma_j\).
If there is exactly one such facet \(G_{k}\), then the intersection \(F\cap G_{k}\) has two components.
Otherwise the intersections \(F\cap G_{k}\) are connected.

If \(\Gamma_j=\Sigma^{n_j}\) with \(n_j>1\), then \(\Gamma_j\) has exactly \(n_j\) facets.
These facets have pairwise non-trivial intersections.
Therefore \(F\cap G_{k}\) is connected if \(j(k)=j\).
And there are exactly \(n_j\) facets of \(Q\) which belong to \(\Gamma_j\).

If \(\Gamma_j=\Delta^{n_j}\) with \(n_j>1\), then \(\Gamma_j\) has exactly \(n_j+1\) facets.
These facets have pairwise non-trivial intersections.
Hence, \(F\cap G_{k}\) is connected if \(j(k)=j\) and there are exactly \(n_j+1\) facets of \(Q\) which belong to \(\Gamma_j\).

  Therefore there are the following cases:
  \begin{enumerate}
  \item \(\dim \Gamma_j=1\) and one of the following statements holds:
    \begin{enumerate}
    \item\label{item:3} There is exactly one facet \(G_{k}\) which belongs
      to \(\Gamma_j\).
    \item\label{item:4} There are exactly two facets \(G_{k_1}\),
      \(G_{k_2}\) which belong to \(\Gamma_j\) and \(G_{k_1}\cap G_{k_2}\neq\emptyset\).
    \item\label{item:5} There are exactly two facets \(G_{k_1}\),
      \(G_{k_2}\) which belong to \(\Gamma_j\) and \(G_{k_1}\cap G_{k_2}=\emptyset\).
    \end{enumerate}
  \item \(\Gamma_j\cong \Sigma^{n_j}\) with \(n_j>1\) and one of the following statements holds:
    \begin{enumerate}
    \item\label{item:6}  There are exactly \(n_j\) facets
      \(G_{k_1},\dots, G_{k_{n_j}}\) which belong to \(\Gamma_j\) and the union of those components of \(\bigcap_{i=1}^{n_j}G_{k_i}\) which meet \(F\) is connected.
    \item\label{item:7} There are exactly \(n_j\) facets
      \(G_{k_1},\dots, G_{k_{n_j}}\) which belong to \(\Gamma_j\) and
      the union of those components of \(\bigcap_{i=1}^{n_j}G_{k_i}\)
      which meet \(F\) is not connected.
    \end{enumerate}
 \item \(\Gamma_j\cong \Delta^{n_j}\) with \(n_j>1\) and one of the following statements holds:
    \begin{enumerate}
    \item\label{item:8}  There are exactly \(n_j+1\) facets
      \(G_{k_1},\dots, G_{k_{n_j+1}}\) which belong to \(\Gamma_j\) and \(\bigcap_{i=1}^{n_j+1}G_{k_i}\neq\emptyset\).
    \item\label{item:9} 
 There are exactly \(n_j+1\) facets
      \(G_{k_1},\dots, G_{k_{n_j+1}}\) which belong to \(\Gamma_j\) and \(\bigcap_{i=1}^{n_j+1}G_{k_i}=\emptyset\).
    \end{enumerate}
  \end{enumerate}

  The proof of the proposition will be completed by the following lemmas.

  \begin{lemma}
    \label{sec:torus-manifolds-with-3}
    If, in the above situation, there is a
    \(j_0\) such that
    \begin{itemize}
    \item \(\Gamma_{j_0}\cong \Sigma^{n_{j_0}}\) with \(n_{j_0}>1\) and the
      union of those components of \(\bigcap_{k;\; j(k)=j_0} G_k\)
      which meet \(F\) is connected, or
    \item \(\Gamma_{j_0}\cong \Delta^{n_{j_0}}\) with \(n_{j_0}>1\) and
      \(\bigcap_{k;\; j(k)=j_0} G_k\neq \emptyset\),
    \end{itemize}
    i.e., one of the cases \ref{item:6} and \ref{item:8} appears at
    \(\Gamma_{j_0}\), then there is an isomorphism of face posets
    \(\mathcal{P}(Q)\rightarrow \mathcal{P}(\bar{\Gamma}_{j_0}\times
    \prod_{i\neq j_0} \Gamma_i)\)  which sends each \(G_{k}\) to a facet
    belonging to the \(j(k)\)-th factor and \(F\) to a facet belonging to
    \(\bar{\Gamma}_{j_0}\). 
  \end{lemma}

  \begin{lemma}
\label{sec:rati-ellipt-torus}
    If, in the above situation, there is a \(j_0\) such that \(\dim \Gamma_{j_0}=1\) and there is exactly one facet
    \(G_{k_0}\) which belongs to \(\Gamma_{j_0}\), i.e.,  the case \ref{item:3} appears at \(\Gamma_{j_0}\), then there
    is an isomorphism of face posets \(\mathcal{P}(Q)\rightarrow
    \mathcal{P}(\Sigma^2\times \prod_{i\neq j_0} \Gamma_i)\), which
    sends each \(G_{k}\) to a facet belonging to the \(j(k)\)-th factor
    and \(F\) to a facet belonging to \(\Sigma^2\).
  \end{lemma}

  \begin{lemma}
\label{sec:rati-ellipt-torus-1}
 If, in the above situation, there is a \(j_0\) such that \(\dim \Gamma_{j_0}=1\) and there are exactly two facets
    \(G_{k_0}\) and \(G_{k_0'}\) which belong to \(\Gamma_{j_0}\) and
    \(G_{k_0}\cap G_{k_0'}\neq \emptyset\), i.e.,
 the case \ref{item:4} appears at \(\Gamma_{j_0}\), then there is
 an isomorphism of face posets
 \(\mathcal{P}(Q)\rightarrow\mathcal{P}(\Delta^2\times \prod_{i\neq
   j_0} \Gamma_i)\) which sends each \(G_{k}\) to a facet belonging
 to the \(j(k)\)-th factor and \(F\) to a facet belonging to
 \(\Delta^2\).
  \end{lemma}

In particular, if one of the cases \ref{item:3}, \ref{item:4}, \ref{item:6} and \ref{item:8} appears at \(\Gamma_{j_0}\), then at the other \(\Gamma_j\) only the cases \ref{item:5}, \ref{item:7} and \ref{item:9} can appear.

  \begin{lemma}
    \label{sec:torus-manifolds-with-4}
    If, in the above situation, at all factors \(\Gamma_j\) only the cases \ref{item:5}, \ref{item:7} and \ref{item:9} appear, then there is an isomorphism of face posets \(\mathcal{P}(Q)\rightarrow \mathcal{P}(F\times [0,1])\) which sends each \(G_{k}\) to \((G_{k}\cap F)\times [0,1]\) and \(F\) to \(F\times \{0\}\).
  \end{lemma} 
\end{proof}

  Now we prove by induction on the dimension of \(Q\) the lemmas from above.
  For \(\dim Q=2\), these lemmas are obvious. Therefore we may assume that \(n=\dim Q>2\) and that all the lemmas are proved in dimensions less than \(n\).

  \begin{sublemma}
\label{sec:rati-ellipt-torus-5}
    Assume that the case \ref{item:6} appears at the factor
    \(\Gamma_{j_0}\) of \(F\).
    Let \(G_{k_0}\) be a facet of \(Q\) which belongs to
    \(\Gamma_{j_0}\).
    Then the following holds:
    \begin{enumerate}
    \item The facets of \(G_{k_0}\) are given by the components of the
      intersections \(G_{k}\cap G_{k_0}\) and \(F\cap G_{k_0}\).
    \item \(G_{k}\cap G_{k_0}\) is connected if and only if \(F\cap
      G_k\) is connected.
    \item There is a combinatorial equivalence
      \begin{equation*}
        \mathcal{P}(G_{k_0})\rightarrow \mathcal{P}(\prod_i \Gamma_i),
      \end{equation*}
such that \(F\cap G_{k_0}\) corresponds to a facet of the \(j_0\)-th
factor and the components of \(G_k\cap G_{k_0}\) correspond to facets
of the \(j(k)\)-th factor.
    \end{enumerate}
  \end{sublemma}
  \begin{proof}
  We consider the inclusion of \(F\cap G_{k_0}\hookrightarrow
  G_{k_0}\), where \(G_{k_0}\) is a facet of \(Q\) which belongs to \(\Gamma_{j_0}\).
  Then \(F\cap G_{k_0}\) is a facet of \(G_{k_0}\) and there is a combinatorial equivalence
  \begin{equation*}
    F\cap G_{k_0}\cong \Sigma^{n_{j_0}-1}\times \prod_{i\neq j_0} \Gamma_i
  \end{equation*}
  such that each component of \(G_{k_0}\cap G_{k}\cap F\) corresponds to a facet of the \(j(k)\)-th factor.
  Moreover, the facets of \(G_{k_0}\) which meet \(F\cap G_{k_0}\) are given by those components of the \(G_{k}\cap G_{k_0}\) which meet \(F\).
  Since the case \ref{item:6} appears at the factor \(\Gamma_{j_0}\) of \(F\) it follows that there is only one component of \(\bigcap_{k;\;j(k)=j_0} G_{k}=\bigcap_{k;\;j(k)=j_0} (G_{k}\cap G_{k_0})\) which meets \(F\).
  Therefore one of the cases \ref{item:6} or \ref{item:3} appears at the factor \(\Sigma^{n_{j_0}-1}\) of \(F\cap G_{k_0}\).

  Hence, it follows from the induction hypothesis that \(G_{k_0}\) is
  combinatorially equivalent to \(\prod_i \Gamma_i\) in such a way
  that the component of \(G_{k_0}\cap G_{k}\) which meets \(F\) is
  mapped to a facet which belongs to the \(j(k)\)-th factor and
  \(F\cap G_{k_0}\) is mapped to a facet which belongs to the
  \(j_0\)-th factor.
  
  Since all facets of \(\prod_i\Gamma_i\) meet the facet which
  corresponds to \(F\cap G_{k_0}\) it follows that each facet of
  \(G_{k_0}\) is a component of some intersection \(G_k\cap G_{k_0}\).

  Moreover, \(G_{k}\cap G_{k_0}\) is connected for all \(k\) with
  \(n_{j(k)}>1\) because for these \(k\) the facets of the factor
  \(\Gamma_{j(k)}\) of \(G_{k_0}\) have pairwise non-trivial intersections.
  If \(n_{j(k)}=1\) and \(j(k)\neq j_0\), then \(G_{k}\cap G_{k_0}\)
  is disconnected if and only if \(G_{k}\cap G_{k_0}\cap F\) is
  disconnected because \(F\cap G_{k_0}\) and the components of
  \(G_k\cap G_{k_0}\) are facets of different factors of
  \(G_{k_0}\cong \prod_i \Gamma_i\).
  Since \(G_k\) and \(G_{k_0}\) belong to different factors of \(F\),
  it follows that this last statement is true if and only if \(G_{k}\cap F\) is disconnected.
 \end{proof}
 
 \begin{sublemma}
\label{sec:rati-ellipt-torus-4}
    Assume that the case \ref{item:8} appears at the factor
    \(\Gamma_{j_0}\) of \(F\).
    Let \(G_{k_0}\) be a facet of \(Q\) which belongs to
    \(\Gamma_{j_0}\).
    Then the following holds:
    \begin{enumerate}
    \item The facets of \(G_{k_0}\) are given by the components of the
      intersections \(G_{k}\cap G_{k_0}\) and \(F\cap G_{k_0}\).
    \item \(G_{k}\cap G_{k_0}\) is connected if and only if \(F\cap
      G_k\) is connected.
    \item There is a combinatorial equivalence
      \begin{equation*}
        \mathcal{P}(G_{k_0})\rightarrow \mathcal{P}(\prod_i \Gamma_i),
      \end{equation*}
such that \(F\cap G_{k_0}\) corresponds to a facet of the \(j_0\)-th
factor and the components of \(G_k\cap G_{k_0}\) correspond to facets
of the \(j(k)\)-th factor.
    \end{enumerate}
  \end{sublemma}
  \begin{proof}
 We consider the inclusions of \(F\cap G_{k_0}\hookrightarrow
  G_{k_0}\) where \(G_{k_0}\) is a facet which belongs to \(\Gamma_{j_0}\).
  Then \(F\cap G_{k_0}\) is a facet of \(G_{k_0}\) and there is a combinatorial equivalence
  \begin{equation*}
    F\cap G_{k_0}\cong \Delta^{n_{j_0}-1} \times \prod_{i\neq j_0} \Gamma_i
  \end{equation*}
  such that each component of \(G_{k_0}\cap G_{k}\cap F\) corresponds to a facet of the \(j(k)\)-th factor.
    
  The facets of \(G_{k_0}\) which meet \(F\cap G_{k_0}\) are given by the components \(C_{ki}\) of \(G_{k}\cap G_{k_0}\) which meet \(F\).  
  The induction hypothesis then implies that \(G_{k_0}\) is combinatorially equivalent to one of the following spaces
  \begin{enumerate}
  \item\label{item:10} \(\Delta^{n_{j_0}}\times \prod_{i\neq j_0} \Gamma_i\) or
  \item\label{item:11} \(\bar{\Gamma}_{j_1}\times\tilde{\Gamma}_{j_0}\times \prod_{i\neq j_1,j_0} \Gamma_i\) or
  \item \([0,1]\times\tilde{\Gamma}_{j_0}\times\prod_{i\neq j_0} \Gamma_i\)
  \end{enumerate}
  in such a way that the \(C_{ki}\) with \(j(k)=j_0\) correspond to facets of the \(j_0\)-th factor.
  In particular, \(G_{k}\cap G_{k_0}=C_{k1}\) is connected for all
  \(k\) with \(j(k)=j_0\).

  Indeed, if \(n_{j_0}>2\), then all facets of
  \(\tilde{\Gamma}_{j_0}\) and \(\Gamma_{j_0}\) have pairwise
  non-trivial intersections. Hence, \(G_{k}\cap G_{k_0}\) has only one
  component in this case.

  If \(n_{j_0}=2\), then \(\tilde{\Gamma}_{j_0}\) has two facets.
  Moreover, the intersection of \(G_{k_0}\) with another facet of
  \(Q\) which belongs to the factor \(\Gamma_{j_0}\) of \(F\) is a non-empty union of
  facets of \(\tilde{\Gamma}_{j_0}\).
  Since besides \(G_{k_0}\) there are two other facets of \(Q\) which
  belong to \(\Gamma_{j_0}\), these intersections must be connected.

  Since \(\bigcap_{k;\;j(k)=j_0} (G_{k}\cap G_{k_0})=\bigcap_{k;\;j(k)=j_0} G_{k}\neq \emptyset\), it follows from the induction hypothesis that we are in case \ref{item:10}.
  In this case an isomorphism of \(\mathcal{P}(G_{k_0})\rightarrow \mathcal{P}(\prod_{i}\Gamma_i)\) is induced by
  \begin{align*}
    F&\mapsto H_{k_0}& C_{ki}&\mapsto H_{ki},
  \end{align*}
  where \(H_{k_0}\) is a facet of the \(j_0\)-th factor
  in the product and the other \(H_{ki}\) are facets of the
  \(j(k)\)-th factor in the product.

  Since all facets of \(\prod_i\Gamma_i\) meet the facet which
  corresponds to \(F\cap G_{k_0}\), it follows that all facets of
  \(G_{k_0}\) are components of intersections \(G_k\cap G_{k_0}\).

  Moreover, \(G_{k}\cap G_{k_0}=C_{k1}\) is connected for all
  \(k\) with \(n_{j(k)}>1\) because all facets of these
  \(\Gamma_{j(k)}\) have pairwise non-trivial intersection.
  If \(n_{j(k)}=1\) and \(j(k)\neq j_0\), then \(G_{k}\cap G_{k_0}\)
  is disconnected if and only if \(G_{k}\cap F\) is disconnected.
  This last statement can be seen as in the proof of Sublemma \ref{sec:rati-ellipt-torus-5}.
 \end{proof}
 
 \begin{sublemma}
\label{sec:rati-ellipt-torus-7}
    Assume that the case \ref{item:3} appears at the factor
    \(\Gamma_{j_0}\) of \(F\).
    Let \(G_{k_0}\) be the facet of \(Q\) which belongs to
    \(\Gamma_{j_0}\).
    Then the following holds:
    \begin{enumerate}
    \item The facets of \(G_{k_0}\) are given by the components of the
      intersections \(G_{k}\cap G_{k_0}\) and \(F\cap G_{k_0}\).
    \item \(G_{k}\cap G_{k_0}\) is connected if and only if \(F\cap
      G_k\) is connected.
    \item There is a combinatorial equivalence
      \begin{equation*}
        \mathcal{P}(G_{k_0})\rightarrow \mathcal{P}(\prod_i \Gamma_i),
      \end{equation*}
such that \(F\cap G_{k_0}\) corresponds to a facet of the \(j_0\)-th
factor and the components of \(G_k\cap G_{k_0}\) correspond to facets
of the \(j(k)\)-th factor.
    \end{enumerate}
  \end{sublemma}
  \begin{proof}
  At first we describe the combinatorial type of \(G_{k_0}\) where
  \(G_{k_0}\) is the facet of \(Q\) which belongs to \(\Gamma_{j_0}\).

  We consider the inclusion of a component \(C\) of \(F\cap G_{k_0}\) in \(G_{k_0}\).
  Then \(C\) is a facet of \(G_{k_0}\) and there is a combinatorial equivalence
  \begin{equation*}
    C\cong \prod_{i\neq j_0} \Gamma_i
  \end{equation*}
  such that each component of \(G_{k_0}\cap G_{k}\cap C\) corresponds to a facet of the \(j(k)\)-th factor.

  Then the facets of \(G_{k_0}\) which meet \(C\) are given by the components \(C_{ki}\) of \(G_{k}\cap G_{k_0}\) which meet \(C\).
  It follows from the induction hypothesis that \(G_{k_0}\) is combinatorially equivalent to one of the following spaces:
  \begin{enumerate}
  \item\label{item:12} \([0,1]\times \prod_{i\neq j_0} \Gamma_i\) or
  \item \(\Sigma^{n_{j_1}+1}\times\prod_{i\neq j_0,j_1} \Gamma_i\) or
  \item \(\Delta^{n_{j_1}+1}\times\prod_{i\neq j_0,j_1} \Gamma_i\),
  \end{enumerate}
  such that each \(C_{ki}\) corresponds to a facet of the \(j(k)\)-th
  factor and \(C\) corresponds to a facet of the first factor.
  By the condition \ref{item:3}, there is a facet of \(G_{k_0}\) which
  does not meet \(C\).
  This facet is the other component of the intersection \(F\cap G_{k_0}\).
    Hence, it follows that we are in case \ref{item:12}.
  
  In this case an isomorphism of \(\mathcal{P}(G_{k_0})\rightarrow \mathcal{P}([0,1]\times \prod_{i\neq j_0} \Gamma_i)\) is induced by
  \begin{align*}
    C&\mapsto H_{k_2}& C'&\mapsto H_{k_2}' &   C_{ki}&\mapsto H_{ki} \text{ for }j(k)\neq j_0,
  \end{align*}
  where \(H_{ki}\) is a facet of the \(j(k)\)-th factor in the
  product, \(H_{k_2}\) and \(H_{k_2}'\) are the facets of the
  \(j_0\)-th factor \([0,1]\) of the product and \(C'\) is the other
  component of \(F\cap G_{k_0}\).
 
  Since all facets of \(\prod_i \Gamma_i\) except the facet
  corresponding to \(C'\) meet the facet corresponding
  to \(C\), it follows that all facets of \(G_{k_0}\) are components
  of intersections \(G_{k}\cap G_{k_0}\).

  As in the proof of Sublemma~\ref{sec:rati-ellipt-torus-5} one sees, moreover, that \(G_{k}\cap G_{k_0}=C_{k}\) is connected for all \(k\) with \(n_{j(k)}>1\).
  If \(n_{j(k)}=1\) and \(j(k)\neq j_0\), then \(G_{k}\cap G_{k_0}\)
  is disconnected if and only if \(G_{k}\cap F\) is disconnected.
  \end{proof}
 
 \begin{sublemma}
\label{sec:rati-ellipt-torus-6}
    Assume that the case \ref{item:4} appears at the factor
    \(\Gamma_{j_0}\) of \(F\).
    Let \(G_{k_0}\) be a facet of \(Q\) which belongs to
    \(\Gamma_{j_0}\).
    Then the following holds:
    \begin{enumerate}
    \item The facets of \(G_{k_0}\) are given by the components of the
      intersections \(G_{k}\cap G_{k_0}\) and \(F\cap G_{k_0}\).
    \item \(G_{k}\cap G_{k_0}\) is connected if and only if \(F\cap
      G_k\) is connected.
    \item There is a combinatorial equivalence
      \begin{equation*}
        \mathcal{P}(G_{k_0})\rightarrow \mathcal{P}(\prod_i \Gamma_i),
      \end{equation*}
such that \(F\cap G_{k_0}\) corresponds to a facet of the \(j_0\)-th
factor and the components of \(G_k\cap G_{k_0}\) correspond to facets
of the \(j(k)\)-th factor.
    \end{enumerate}
  \end{sublemma}
  \begin{proof}
   Let \(G_{k_0}\) and \(G_{k_0'}\) be the two facets of \(Q\), which belong to \(\Gamma_{j_0}\).

Then the
intersection \(G_{k_0}\cap G_{k_0'}\) is
non-empty.
Therefore \(G_{k_0}\) has a facet which is not equal to \(F\cap
G_{k_0}\) or a component of \(G_{k}\cap G_{k_0}\) with \(j(k)\neq j_0\).
Therefore, as in the proof of Sublemma \ref{sec:rati-ellipt-torus-7},
one sees that \(G_{k_1}\) is combinatorially equivalent to
\([0,1]\times \prod_{i\neq j_0} \Gamma_i\).

The other statements of the Sublemma can be seen as in the proof of Sublemma~\ref{sec:rati-ellipt-torus-5}.
  \end{proof}

  Next we show that, if there is a factor \(\Gamma_{j_0}\) of \(F\)
  where one of the cases \ref{item:3}, \ref{item:4}, \ref{item:6} or
  \ref{item:8} occurs, then at the other factors of \(F\) only the cases
  \ref{item:5},~\ref{item:7} or~\ref{item:9} can occur.

\begin{sublemma}
\label{sec:rati-ellipt-torus-11}
  There is at most one factor \(\Gamma_{j_0}\) of \(F\), where one of the cases
  \ref{item:3}, \ref{item:4}, \ref{item:6} or \ref{item:8} appears.
\end{sublemma}
\begin{proof}
  Assume that one of the cases \ref{item:6} and \ref{item:8} occurs at the factor \(\Gamma_{j_0}\) of \(F\) and one of the cases  \ref{item:3}, \ref{item:4}, \ref{item:6} and \ref{item:8} appears at another factor \(\Gamma_{j_1}\) of \(F\).
  Then we consider the intersection \(G_{k_0}\cap G_{k_1}\) with
  \(j(k_0)=j_0\) and \(j(k_1)=j_1\).
  It follows from the description of the combinatorial type of
  \(G_{k_0}\) given in the Sublemmas~\ref{sec:rati-ellipt-torus-5} and
   \ref{sec:rati-ellipt-torus-4}, that there is an isomorphism of posets
  \begin{equation*}
    \mathcal{P}(C)\rightarrow\mathcal{P}(\Gamma_{j_0}\times \tilde{\Gamma}_{j_1}\times \prod_{i\neq j_0,j_1} \Gamma_i),
  \end{equation*}
  where \(C\) is a component of \(G_{k_1}\cap G_{k_0}\) such that \(F\cap G_{k_0}\cap G_{k_1}\) and the \(G_{k}\cap G_{k_0}\cap G_{k_1}\), \(j(k)=j_0\), correspond to the facets belonging to the factor \(\Gamma_{j_0}\).
  Hence, it follows that the intersection of \(G_{k_0}\cap G_{k_1}\) with \(\bigcap_{k;\;j(k)=j_0} G_{k}\) is non-empty (or connected) if case \ref{item:8} (or \ref{item:6}, respectively) appears at the factor \(\Gamma_{j_0}\) of \(F\).

  From the description of the combinatorial type of \(G_{k_1}\) given
  in Sublemmas~\ref{sec:rati-ellipt-torus-5},
  \ref{sec:rati-ellipt-torus-4}, \ref{sec:rati-ellipt-torus-7} and \ref{sec:rati-ellipt-torus-6} it follows that there is an isomorphism of posets
  \begin{equation*}
    \mathcal{P}(G_{k_0}\cap G_{k_1})\rightarrow\mathcal{P}(\tilde{\Gamma}_{j_0}\times \Gamma_{j_1}\times \prod_{i\neq j_0,j_1} \Gamma_i),
  \end{equation*}
  such that the  \(G_{k}\cap G_{k_0}\cap G_{k_1}\), \(j(k)=j_0\), correspond to the facets belonging to the factor \(\tilde{\Gamma}_{j_0}\).
  Hence, it follows that the intersection of  \(G_{k_0}\cap G_{k_1}\) with \(\bigcap_{k;\;j(k)=j_0} G_{kj_0}\) is empty (or non-connected) if case \ref{item:8} (or \ref{item:6}, respectively) appears at the factor \(\tilde{\Gamma}_{j_0}\) of \(F\).
  Therefore we have a contradiction.

  Next assume that at one factor \(\Gamma_{j_0}\) of \(F\) the case \ref{item:3} appears and at
another factor \(\Gamma_{j_1}\) the case \ref{item:4} appears.
Then it follows from the description of the combinatorial type of
\(G_{k_0}\) given in Sublemma~\ref{sec:rati-ellipt-torus-7} that the intersection \(G_{k_0}\cap G_{k_1}\) is
connected.
Here \(G_{k_0}\) is the facet belonging to the factor \(\Gamma_{j_0}\)
and \(G_{k_1}\) is a facet belonging to the factor \(\Gamma_{j_1}\).
But the description of the combinatorial type of \(G_{k_1}\)
given in Sublemma~\ref{sec:rati-ellipt-torus-6} implies that this intersection is disconnected.

Next assume that the case \ref{item:4} occurs at two factors
\(\Gamma_{j_0}\) and \(\Gamma_{j_1}\).
Let \(G_{k_0}\) and \(G_{k_0'}\) be the facets belonging to
\(\Gamma_{j_0}\).
Moreover, let \(G_{k_1}\) be a facet of \(Q\) belonging to \(\Gamma_{j_1}\).
Then it follows from the description of the combinatorial type of
\(G_{k_0}\) given in Sublemma \ref{sec:rati-ellipt-torus-6} that the intersection \(G_{k_0}\cap G_{k_0'}\cap
G_{k_1}\) is non-empty.
But from the description of the combinatorial type of \(G_{k_1}\) it
follows that this intersection is empty.

Last we show that the case \ref{item:3} occurs at
at most one factor of \(F\).
Assume that the case \ref{item:3} appears at the factors
\(\Gamma_{j}\) with \(j<s\) and not at the factors \(\Gamma_{j}\) with
\(j\geq s\).
Then for each \(j\geq s\) choose \(n_j\) facets \(G_{k_{j1}},\dots,G_{k_{jn_j}}\)   of \(Q\) belonging to
that factor.
Consider a component \(K\) of \(\bigcap_{j\geq s} \bigcap_{i=1}^{n_j} G_{k_{ji}
}\) which meets \(F\).
Since at none of the factors \(\Gamma_{j}\), \(j\geq s\), the cases
\ref{item:3} or \ref{item:6} appear, the intersection \(F\cap K\) is
connected.
Therefore also the intersections \(G_k\cap K\), \(j(k)<s\) are
connected by Sublemma~\ref{sec:rati-ellipt-torus-7}.

By the description of the combinatorial type of the \(G_k\) given in
Sublemma~\ref{sec:rati-ellipt-torus-7} the intersections \(G_k\cap K\)
are combinatorially equivalent to cubes
and have pairwise non-connected intersection.
Therefore it follows from Lemma \ref{sec:rati-ellipt-torus-3}, that
\(K\) has dimension two.
Hence it follows that \(s\) is two, i.e. there is only one factor of
\(F\) where the case \ref{item:3} appears.

  Therefore there is at most one factor \(\Gamma_{j_0}\) of \(F\) where one of the cases  \ref{item:3}, \ref{item:4}, \ref{item:6} and \ref{item:8} appears.
\end{proof}

  \begin{sublemma}
\label{sec:rati-ellipt-torus-8}
    Assume that one of the cases \ref{item:3}, \ref{item:4},
    \ref{item:6} or \ref{item:8} appears at the factor
    \(\Gamma_{j_0}\) of \(F\).
    Let \(G_{k_1}\) be a facet of \(Q\) which meets \(F\) and belongs
    to \(\Gamma_{j_1}\) with \(j_1\neq j_0\).
Then the following holds:
\begin{enumerate}
\item The facets of \(G_{k_1}\) are given by the components of the intersections \(G_k\cap G_{k_1}\).
\item \(G_{k_1}\cap G_k\) is disconnected if and only if
  \(\Gamma_{j_1}=\Sigma^2\) and \(j(k)=j(k_1)=j_1\). 
\item There is a combinatorial equivalence
  \begin{equation*}
    G_{k_1}\cong \bar{\Gamma}_{j_0}\times \tilde{\Gamma}_{j_1}\times
    \prod_{i\neq j_0,j_1} \Gamma_i,
  \end{equation*}
such that \(F\cap G_{k_1}\) corresponds to a facet of the \(j_0\)-th
factor and the \(G_k\cap G_{k_1}\) correspond to disjoint unions of
facets of the \(j(k)\)-th factor.
\end{enumerate}
  \end{sublemma}
  \begin{proof} 
  We consider the inclusion \(F\cap G_{k_1}\hookrightarrow G_{k_1}\).
  \(F\cap G_{k_1}\) is a facet of \(G_{k_1}\) and there is a
  combinatorial equivalence
  \begin{equation*}
    G_{k_1}\cap F\cong \tilde{\Gamma}_{j_1}\times\prod_{i\neq j_1} \Gamma_i
  \end{equation*}
  such that each component of \(G_{k}\cap G_{k_1}\cap F\) corresponds to a facet of the \(j(k)\)-th factor.
  Moreover, the facets of \(G_{k_1}\) which meet \(G_{k_1}\cap F\) are given by the components \(C_{ki}\) of \(G_{k}\cap G_{k_1}\) which meet \(G_{k_1}\cap F\).
 
 If \(j(k)=j_0\), then \(G_k\cap F\) and \(F\cap G_{k_1}\) are
 facets of different factors of 
 \(F\).
 Since \(G_k\cap G_{k_0}\) and \(G_{k_0}\cap
 G_{k_1}\) are facets of different factors of \(G_{k_0}\cong\prod_i\Gamma_i\) and because \(F\cap G_{k_0}\) is a
 facet of the \(j_0\)-th factor of \(G_{k_0}\), it follows from Sublemma~\ref{sec:rati-ellipt-torus-4} that
 \(\bigcap_{k;\;j(k)=j_0} (G_{k}\cap G_{k_1})\neq \emptyset\) if the case \ref{item:8} appears at the factor \(\Gamma_{j_0}\) of \(F\).

Furthermore, by the same argument \(\bigcap_{k;\;j(k)=j_0} (G_{k}\cap
G_{k_1})\neq \emptyset\) is connected if the case \ref{item:6} appears
at the factor \(\Gamma_{j_0}\) of \(F\). Here one uses Sublemma~\ref{sec:rati-ellipt-torus-5}.

If the case \ref{item:3} appears at the factor \(\Gamma_{j_0}\) one
can argue as follows.
  It follows from the description of the combinatorial type of
  \(G_{k_0}\), \(j(k_0)=j_0\), given in Sublemma~\ref{sec:rati-ellipt-torus-7} that 
  \(G_{k_0}\cap G_{k_1}\) is connected.
  Moreover, it follows from the same description that  \(G_{k_0}\cap
  G_{k_1}\cap F\)
  has two components.
 
  Therefore  the case \ref{item:3} also
  appears at the factor \(\Gamma_{j_0}\) of \(G_{k_1}\cap F\).

If the case \ref{item:4} appears at the factor \(\Gamma_{j_0}\), then
it follows from the combinatorial description of the \(G_k\) with
\(j(k)=j_0\) given in Sublemma \ref{sec:rati-ellipt-torus-6} that for these \(k\) the intersection of \(G_k\) with
\(G_{k_1}\) is connected.
Moreover, the intersection \(\bigcap_{k;\;j(k)=j_0} G_k\cap G_{k_1}\) is non-empty.
Therefore the case \ref{item:4} appears at the factor \(\Gamma_{j_0}\)
of \(G_{k_1}\cap F\).

Hence, the case which appears for the factor \(\Gamma_{j_0}\) of \(F\)
also appears at the factor \(\Gamma_{j_0}\) of \(G_{k_1}\cap F\).
  Therefore from the induction hypothesis we get an isomorphism of posets 
  \begin{equation*}
    \mathcal{P}(G_{k_1})\rightarrow \mathcal{P}(\bar{\Gamma}_{j_0}\times \tilde{\Gamma}_{j_1}\times \prod_{i\neq j_0,j_1} \Sigma^{n_{i}}\times \prod_{i\neq j_0j_1} \Delta^{n_i}),
  \end{equation*}
  such that \(G_{k_1}\cap F\) is mapped to a facet of the \(j_0\)-th
  factor and \(C_{ki}\) to a facet of the \(j(k)\)-th factor. Here
  \(C_{ki}\) is a compoenent of \(G_k\cap G_{k_1}\).

  Since all pairs of facets of \(\bar{\Gamma}_{j_0}\) have non-trivial
  intersection, all facets of \(G_{k_1}\) meet \(G_{k_1}\cap F\).
  Moreover, it follows that \(G_{k}\cap G_{k_1}\) is connected if \(j(k)\neq j_1\) or \(\Gamma_{j_1}\neq \Sigma^{2}\).
  Otherwise this intersection has two components.

  Indeed, if \(j(k_2)\neq j_1,j_0\), then it follows from the description
  of the combinatorial type of \(F\) that \(F\cap G_{k_1}\cap
  G_{k_2}=(F\cap G_{k_1})\cap (F\cap G_{k_2})\) is connected.
  Because \(G_{k_2}\) and \(F\) belong to different factors of
  \(G_{k_1}\) it follows that \(G_{k_2}\cap G_{k_1}\) is connected.

  Next assume that \(j(k_2)=j_1\) and \(\dim \Gamma_{j_1} \geq 3\).
  Then all pairs of facets of \(\tilde{\Gamma}_{j_1}\) have
  non-trivial intersections. Therefore \(G_{k_2}\cap G_{k_1}\) is
  connected in this case.

  Assume now that \(\Gamma_{j_1}=\Delta^2\).
  Then besides \(G_{k_1}\) there are two other facets of \(Q\) which
  belong to \(\Gamma_{j_1}\). These two facets have non-trivial
  intersections with \(G_{k_1}\).
  Moreover, the components of these intersections are facets of the
  factor \(\tilde{\Gamma}_{j_1}\) of \(G_{k_1}\).
  Since \(\tilde{\Gamma}_{j_1}\) has two facets, the intersections
  \(G_{k_2}\cap G_{k_1}\) with \(j(k_2)=j_1\) are connected.

  Next assume that \(\Gamma_{j_1}=\Sigma^2\).
  Then besides \(G_{k_1}\) there is exactly one other facet
  \(G_{k_2}\) of \(Q\) which belongs to \(\Gamma_{j_1}\).
  Moreover, \(F\cap G_{k_2}\cap G_{k_1}\) has two components.
  Since \(F\cap G_{k_1}\) and the components of \(G_{k_2}\cap
  G_{k_1}\) are facets of different factors of \(G_{k_1}\).
  It follows that \(G_{k_2}\cap G_{k_1}\) has two components.

  At last assume that \(\dim \Gamma_{j_1}=1\).
  Then since \(F\cap G_{k_1}\) is connected. There is another facet
  \(G_{k_2}\) of \(Q\) which belongs to \(\Gamma_{j_1}\).
  Since \(F\cap G_{k_1}\cap G_{k_2}\) is empty and all facets of
  \(G_{k_1}\) meet \(F\cap G_{k_1}\) it follows that \(G_{k_1}\cap
  G_{k_2}\) is empty.
  \end{proof}

Now we can prove the Lemmas \ref{sec:torus-manifolds-with-3},
\ref{sec:rati-ellipt-torus}, and \ref{sec:rati-ellipt-torus-1}. 

\begin{proof}[Proof of Lemmas \ref{sec:torus-manifolds-with-3},
\ref{sec:rati-ellipt-torus} and \ref{sec:rati-ellipt-torus-1}]
  For \(j\neq j_0\), let \(\tilde{n}_j=n_j\) and \(\tilde{n}_{j_0}=n_{j_0}+1\).
  Moreover, let \(G_{0}=F\) and \(j(0)=j_0\).
  Let \(P=\prod_{i<r} \Sigma^{\tilde{n}_i}\times \prod_{i\geq r}
  \Delta^{\tilde{n}_i}\).
  Denote by \(H_{k}\) the facets of \(P\).
  We have shown in Sublemmas~\ref{sec:rati-ellipt-torus-5},
  \ref{sec:rati-ellipt-torus-4}, \ref{sec:rati-ellipt-torus-7},
  \ref{sec:rati-ellipt-torus-6} and \ref{sec:rati-ellipt-torus-8} that there are isomorphisms of posets
  \begin{equation*}
    \mathcal{P}(G_{k})\rightarrow \mathcal{P}(H_{k}) 
  \end{equation*}
  such that \((G_{k}\cap G_{k'})\mapsto (H_{k}\cap H_{k'})\) where 
  \(H_k\) and \(H_{k'}\) are facets of the \(j(k)\)-th and
  \(j(k')\)-th factor of \(P\), respectively.
  
  If \(\bigcap_{k \in K} H_{k}\neq \emptyset\), then this intersection has \(2^m\) components, where \(m\) is the number of \(j_1\)'s with \(j_1<r\) and \(K\supset I_{j_1}=\{k;\; j(k)=j_1\}\).
  If \(K=I_{j_1}\) as above, then \(\bigcap_{k\in K} H_{k}\) has two components \(C_1\) and \(C_2\).
  Moreover, there is an automorphism of \(\mathcal{P}(P)\), which interchanges \(C_1\) and \(C_2\) and fixes all faces of \(P\) not contained in \(C_1\cup C_2\).
  Therefore, after composing some of the isomorphisms \(\mathcal{P}(G_{k})\rightarrow \mathcal{P}(H_{k})\) with these automorphisms if necessary, we can extend these isomorphisms to an isomorphism
  \begin{equation*}
    \mathcal{P}(Q)\rightarrow \mathcal{P}(P),
  \end{equation*}
  with \(G_{k}\mapsto H_{k}\).
  This completes the proof of the lemmas.
\end{proof}

For the proof of Lemma~\ref{sec:torus-manifolds-with-4} we need two
more sublemmas.

\begin{sublemma}
\label{sec:rati-ellipt-torus-10}
  Assume that we are in the situation of
  Lemma~\ref{sec:torus-manifolds-with-4}. Let \(G_{k_0}\) be a
  facet of \(Q\) belonging to the factor \(\Gamma_{j_0}\) of \(F\).
  Then there is a combinatorial equivalence
  \(\mathcal{P}(G_{k_0})\rightarrow \mathcal{P}((F\cap
  G_{k_0})\times[0,1])\) which sends each \(G_k\cap G_{k_0}\) to
  \((F\cap G_{k_0}\cap G_k)\times [0,1]\) and \(F\cap G_{k_0}\) to \((F\cap
  G_{k_0})\times \{0\}\).
\end{sublemma}
\begin{proof}
   We consider the inclusion \(F\cap G_{k_0}\hookrightarrow G_{k_0}\).
  Then \(F\cap G_{k_0}\) is a facet of \(G_{k_0}\) and there is a combinatorial equivalence
  \begin{equation*}
     F\cap G_{k_0}\cong \tilde{\Gamma}_{j_0}\times \prod_{i\neq j_0}\Gamma_i
  \end{equation*}
 such that each component of \(G_{k_0}\cap G_{k}\cap F\) corresponds to a facet of the \(j(k)\)-th factor.
  The facets of \(G_{k_0}\) which meet \(F\cap G_{k_0}\) are given by  the components of \(G_{k_0}\cap G_{k}\) which meet \(F\).
  We show that if one of the cases  \ref{item:5}, \ref{item:7} and \ref{item:9} appears at the factor \(\Gamma_j\) of \(F\) then the same holds for the factor \(\Gamma_j\) of \(F\cap G_{k_0}\).
  If one of the cases \ref{item:5} and \ref{item:9} appears this is clear because
  \begin{equation*}
    \bigcap_{k;\;j(k)=j}(G_{k}\cap G_{k_0})\subset  \bigcap_{k;\;j(k)=j}G_k =\emptyset.
  \end{equation*}
  Therefore assume that case \ref{item:7} occurs at \(\Gamma_j\).
 
  At first assume that \(j(k)=j_0\) and \(n_j=2\).
  Then there is only one \(G_{k}\) which belongs to \(\Gamma_{j_0}\)
  and is not equal to \(G_{k_0}\).
  The intersection of \(G_{k}\) and \(G_{k_0}\) has two components which meet \(F\) because  \(G_{k}\cap G_{k_0}\cap F\) has two components and the union of those components of \(G_{k}\cap G_{k_0}\) which meet \(F\) is disconnected.
  Therefore we have that case \ref{item:5} appears at the factor \(\tilde{\Gamma}_{j_0}\) of \(F\cap G_{k_0}\).
  
  Next assume that \(j(k)=j\neq j_0\) or \(n_{j_0}>2\).
  Then there is only one component \(B_{k}\)  of \(G_{k_0}\cap G_{k}\) which meets \(F\) because \(G_{k_0}\cap G_{k}\cap F\) is connected.
  Clearly \(\bigcap_{k;\;j(k)=j} B_{k}\) is contained in \(\bigcap_{k;\;j(k)=j} G_{k}\cap G_{k_0}\).
  By dimension reasons \(\bigcap_{k;\;j(k)=j} B_{k}\) is a union of components of \(\bigcap_{k;\;j(k)=j} G_{k}\cap G_{k_0}\).
  In fact, the union of those components of \(\bigcap_{k;\;j(k)=j} B_{k}\) which meet \(F\) is equal to the union of those components of \(\bigcap_{k;\;j(k)=j} G_{k}\cap G_{k_0}\) which meet \(F\).

  Since every component of \(\bigcap_{k;\;j(k)=j} G_{k}\cap F\) contains exactly one component of \(\bigcap_{k;\;j(k)=j} G_{k}\cap G_{k_0}\cap F\), each component of \(\bigcap_{k;\;j(k)=j} G_{k}\) which meets \(F\) contains a component of \(\bigcap_{k;\;j(k)=j} B_{k}\) which meets \(F\).
  Because the union of those components of \(\bigcap_{k;\;j(k)=j} G_{k}\) which meet \(F\) is disconnected, the same holds for the union of those components of \(\bigcap_{k;\;j(k)=j} B_{k}\) which meet \(F\). 
  Hence, the case \ref{item:7} appears at the factor \(\Gamma_j\) of \(F\cap G_{k_0}\).

  Therefore it follows from the induction hypotheses that there is an isomorphism of posets \(\mathcal{P}(G_{k_0})\rightarrow\mathcal{P}((F\cap G_{k_0})\times [0,1])\), such that \(F\) is mapped to \((F\cap G_{k_0})\times \{0\}\) and the component \(C\) of \(G_{k_0}\cap G_{k}\) which meets \(F\) is mapped to \((F\cap C)\times [0,1]\).
  In particular, all components of \(G_{k_0}\cap G_{k}\) meet \(F\).
\end{proof}

\begin{sublemma}
\label{sec:rati-ellipt-torus-9}
  In the situation of Lemma \ref{sec:torus-manifolds-with-4}, there
  is exactly one facet \(H\) of \(Q\) which does not meet
  \(F\). Moreover, the following holds:
  \begin{enumerate}
  \item Under the isomorphism constructed in the previous sublemma,
    \(G_{k_0}\cap H\) corresponds to \((F\cap G_{k_0})\times \{1\}\).
  \item \(\mathcal{P}(H)\cong \mathcal{P}(F)\).
  \end{enumerate}
\end{sublemma}
  \begin{proof}
   Let \(H_{k_0}\) be the facet of \(Q\) which intersects \(G_{k_0}\) in the facet of \(G_{k_0}\) which corresponds to \((F\cap G_{k_0})\times \{1\}\).
  If \(q\) is a vertex of \(G_{k_0}\) corresponding to \((p,1)\), where \(p\) is a vertex of \(F\cap G_{k_0}\), then \(H_{k_0}\) is the facet of \(Q\) which is perpendicular to the edge of \(G_{k_0}\) which corresponds to \(\{p\}\times [0,1]\). 
  We claim that all \(H_{k_0}\) are the same.

  If \(j_1\neq j_2\), then the intersection of \(G_{k_1}\cap F\) and
  \(G_{k_2}\cap F\), where \(j(k_1)=j_1\) and \(j(k_2)=j_2\), is non-empty.
  If \(j_1=j_2\) and \(\dim G_{k_1}\cap F=\dim G_{k_2}\cap F>0\), there is a \(G_{k'}\) such that \(G_{k'}\cap G_{k_i}\cap F\neq\emptyset\) for \(i=1,2\).
  Hence, in these cases \(H_{k_1}=H_{k_2}\) because there is a vertex in \(G_{k_1}\cap G_{k_2}\cap F\).
  Since \(\dim Q>2\), \(F\) cannot be an interval.
  Hence, it follows that all \(H_{k}\) are equal, so that we can drop the indices.
  Since the vertex-edge-graph of \(H\) is connected, every face of \(Q\) contains at least one vertex and each vertex is contained in exactly \(n-1\) facets of \(H\), it follows that the facets of \(H\) are given by the \(G_{k}\cap H\).
  
  Indeed, if there is another facet of \(H\), then it contains a vertex \(v\) of \(H\).
  Since the vertex-edge-graph of \(H\) is connected, we may assume that \(v\) is connected by an edge to a vertex \(v'\in G_{k}\cap H\).
  It follows from the description of the combinatorial type of \(G_{k}\) that \(v'\) is contained in \(n-1\) facets of \(H\) of the form \(G_{k'}\cap H\).
  Therefore each edge which meets \(v'\) is contained in a facet of the form \(G_{k'}\cap H\).
  Hence, \(v\in G_{k'}\).
  Therefore it follows from the description of the combinatorial type of \(G_{k'}\) that all facets of \(H\) which contain \(v\) are of the form  \(G_{k''}\cap H\).
  This is a contradiction to the assumption that \(v\) is contained in a facet which is not of this form.

  Therefore there is an isomorphism of posets \(\phi: \mathcal{P}(F)\rightarrow \mathcal{P}(H)\), such that \(\phi(C_K\cap F)=C_K\cap H\).
  Here \(C_K\) is a component of the intersection \(\bigcap_{k\in K} G_{k}\).
  Since the vertex-edge-graph of \(Q\) is connected, every face of \(Q\) contains at least one vertex and each vertex is contained in exactly \(n\) facets, \(F\), \(H\), \(G_{k}\) is a complete list of facets of \(Q\).
\end{proof}

\begin{proof}[Proof of Lemma~\ref{sec:torus-manifolds-with-4}] 
  It follows from Sublemmas~\ref{sec:rati-ellipt-torus-10} and \ref{sec:rati-ellipt-torus-9} that the face posets of \(Q\) and \(F\times [0,1]\) are isomorphic. An isomorphism is given by
  \begin{align*}
    C_K&\mapsto (C_K\cap F)\times [0,1]& (F\cap C_K)&\mapsto (F\cap
    C_K)\times \{0\} \\ & & (H\cap C_K)&\mapsto (F\cap C_K)\times\{1\}.
  \end{align*}
  Here \(C_K\) is a component of the intersection \(\bigcap_{k\in K} G_{k}\).
  Therefore the sublemma is proved.
  \end{proof}

\begin{proof}[Proof of Theorem~\ref{sec:torus-manifolds-with}]
  It follows from Lemmas \ref{sec:torus-manifolds-with-1} and \ref{sec:torus-manifolds-with-2} that \(M/T\) is combinatorially equivalent to \(P=\prod_{i<r}\Sigma^{n_i}\times\prod_{i\geq r}\Delta^{n_i}\).
  Therefore, by Theorem \ref{sec:simpl-torus-acti-3}, \(M\) is homeomorphic to a torus manifold \(M'\) over \(P\).
  The manifold \(M'\) can be constructed as the model \(M_P(\lambda)\), where \(\lambda\) is the characteristic map of \(M\).
Now \(M'\) is the quotient of a free torus action on the moment angle complex \(Z_P\) associated to \(P\).
  But \(Z_P\) is equivariantly homeomorphic to a product of spheres with linear torus action.
  Therefore the theorem is proved.
\end{proof}

\section{Applications to rigidity problems in toric topology}
\label{sec:rigidity}

A torus manifold \(M\) is called quasitoric if it is locally standard and \(M/T\) is face-preserving homeomorphic to a simple convex polytope.
In toric topology there are two notions of rigidity one for simple polytopes and one for quasitoric manifolds.
These are:

\begin{definition}
  Let \(M\) be a quasitoric manifold over the polytope \(P\).
  \begin{itemize}
  \item \(M\) is called rigid if any other quasitoric manifold \(N\) with \(H^*(N;\mathbb{Z})\cong H^*(M;\mathbb{Z})\) is homeomorphic to \(M\).
  \item \(P\) is called rigid if any other simple polytope \(Q\), such that it exists a quasitoric manifold \(N\) over \(Q\) and a quasitoric manifold \(M'\) over \(P\) with \(H^*(N;\mathbb{Z})\cong H^*(M';\mathbb{Z})\), is combinatorially equivalent to \(P\).
  \end{itemize}
\end{definition}

It has been shown  by Choi, Panov and Suh \cite{MR2725043} that a product of simplices is a rigid polytope.
As a consequence of Theorem~\ref{sec:torus-manifolds-with} we have the following partial generalization of their result.

\begin{theorem}
\label{sec:appl-rigid-probl}
  Let \(M_1\) and \(M_2\) be two simply connected torus manifolds with
  \(H^*(M_1;\mathbb{Q})\cong H^*(M_2;\mathbb{Q})\) and \(H^{\text{odd}}(M_i;\mathbb{Z})=0\).
  Assume that \(\mathcal{P}(M_1/T)\) is isomorphic to \(\mathcal{P}(\prod_i \Sigma^{n_i}\times \prod_i \Delta^{n_i})\).
  Then the face posets of the orbit spaces of \(M_1\) and \(M_2\) are isomorphic.
\end{theorem}
\begin{proof}
By Theorem~\ref{sec:simpl-torus-acti-3}, \(M_1\) is homeomorphic to a
quotient of a free linear torus action on a product of spheres.
Since the cohomology of such a quotient is intrinsically formal,
\(M_1\) and \(M_2\) are rationally homotopy equivalent and rationally elliptic.

Therefore both \(M_1\) and \(M_2\)
are homeomorphic to quotients of free torus actions on products of spheres \(S_i\), \(i=1,2\), where the dimension of the acting torus \(T_i\) is equal to the number of odd dimensional spheres in the product.
  Moreover, each factor in these products has at least dimension \(3\).
  And each factor in \(S_i\) corresponds to a factor of the face-poset of \(M_i/T\) which is combinatorially equivalent to \(\prod_j \Sigma^{n_{ji}}\times \prod_j \Delta^{n_{ji}}\).
  Therefore we have
  \begin{align*}
    \dim \pi_2(M_i)\otimes\Q&= \dim T_i&
    \dim \pi_2(S_i)\otimes\Q&= 0\\
    \dim \pi_j(M_i)\otimes\Q&= \dim \pi_j(S_i)\otimes\Q
  \end{align*}
for \(i=1,2\) and \(j>2\).
Since two products of spheres have the same rational homotopy groups if and only if they have the same number of factors of each dimension, it follows that the face posets of \(M_1\) and \(M_2\) are isomorphic.
\end{proof}

It is known that \(\prod_i \C P^{n_i}\), is rigid among quasitoric manifolds (see \cite{MR2978415} and the references therein).
The next corollary shows that \(\prod_i \C P^{n_i}\) is rigid among simply connected torus manifolds.

\begin{cor}
\label{sec:appl-rigid-probl-1}
  Let \(M\) be a simply connected torus manifold with \(H^*(M;\mathbb{Z})\cong H^*(\prod_i \C P^{n_i};\mathbb{Z})\).
  Then \(M\) is homeomorphic to \(\prod_i \C P^{n_i}\).
\end{cor}
\begin{proof}
  By Theorem~\ref{sec:appl-rigid-probl}, we know that \(\mathcal{P}(M/T)\) is isomorphic to \(\mathcal{P}(\prod_i \Delta^{n_i})\).
  Denote by \(\lambda\) the characteristic function of \(M\).
  Then from a canonical model we can construct a quasitoric manifold \(M_1\) over \(\prod_i \Delta^{n_i}\) with characteristic function \(\lambda\).
  By Theorem~\ref{sec:simpl-torus-acti-3}, \(M\) and \(M_1\) are homeomorphic.
  Moreover, by Corollary 1.3 of \cite{MR2978415}, \(M_1\) is homeomorphic to \(\prod_i \C P^{n_i}\).
  Therefore the corollary follows.
\end{proof}

\section{Non-negatively curved torus manifolds}
\label{sec:towards}

In this section we prove the following:

\begin{theorem}
\label{sec:proof-theor-refs-1}
  Let \(M\) be a simply connected non-negatively curved torus manifold.
  Then \(M\) is equivariantly diffeomorphic to a quotient of a free linear torus action on a product of spheres.
\end{theorem}

For the proof of this theorem we need the following result of Spindeler.

\begin{theorem}[{\cite[Theorem 3.28 and Lemma 3.30]{spindeler}}]
\label{sec:proof-theor-refs}
  Let \(M\) be a closed non-negatively curved fixed point homogeneous Riemannian manifold.
  Then for every maximal fixed point component \(F\) there exists a smooth invariant submanifold  \(N\subset M\) such that \(M\) decomposes as the union of the normal disc bundles of \(N\) and \(F\):
  \begin{equation}
    \label{eq:1} M\cong D(F)\cup_E D(N).
  \end{equation}
Here \(E=\partial D(F)\cong \partial D(N)\).
Further \(N\) is invariant under the group \(U=\{f\in\Iso(M);\; f(F)=F\}\). Moreover, the decomposition~(\ref{eq:1}) is \(U\)-equivariant with respect to the natural action of \(U\) on \(D(F)\), \(D(N)\) and \(M\).
\end{theorem}

Here a Riemannian \(G\)-manifold is called fixed point homogeneous if there is a component \(F\) of \(M^G\), such that, for every \(x\in F\), \(G\) acts transitively on the normal sphere \(S(N_x(F,M))\).
Such a component \(F\) is called maximal component of \(M^G\).

The above mentioned natural \(U\)-actions on the normal disc bundles are
given by the restrictions of the natural actions on the normal bundles given by
differentiating the original action on \(M\).

Now let \(M\) be a torus manifold and \(F\subset M\) a characteristic submanifold.
Then \(M\) is naturally a fixed point homogeneous manifold with respect to the \(\lambda(F)\)-action on \(M\).
Moreover, the torus \(T\) is contained in the group \(U\) from the above theorem.
In this situation we have the following lemma.

\begin{lemma}
\label{sec:proof-theor-refs-2}
  Let \(M\) be a  simply connected torus manifold with an invariant metric of non-negative curvature.
  Then \(M\) is locally standard and \(M/T\) and all its faces are diffeomorphic (after-smoothing the corners) to standard discs \(D^k\).
  Moreover, \(H^{\text{odd}}(M;\mathbb{Z})=0\).
\end{lemma}
\begin{proof}
  We prove this lemma by induction on the dimension of \(M\).
  If \(2n=\dim M\leq 2\), then this is obvious.
  Therefore assume that \(\dim M\geq 4\) and that the lemma is proved in all dimensions less than \(\dim M\).

  By Theorem~\ref{sec:proof-theor-refs}, we have a decomposition
  \begin{equation*}
    M=D(N)\cup_E D(F),
  \end{equation*}
where \(F\) is a characteristic submanifold of \(M\) and \(E\) the
\(S^1\)-bundle associated to the normal bundle of \(F\).
Spindeler proved that \(\codim N\geq 2\) and \(\pi_1(F)=0\) if \(M\) is simply connected \cite[Lemma 3.29 and Theorem 3.35]{spindeler} (see also the proof of Lemma~\ref{sec:non-negat-curv-1} below).

Since \(F\) is totally geodesic in \(M\), it admits an invariant metric of non-negative curvature.

It follows from the exact homotopy sequence for the fibration \(\pi_F:E\rightarrow F\) that \(\pi_1(E)\) is cyclic and generated by the inclusion of a fiber of \(\pi_F\).

The circle subgroup \(\lambda(F)\) of \(T\), which fixes \(F\), acts
freely on \(E\) by multiplication on the fibers of \(\pi_F\).
It follows from the exact homotopy sequence for the fibration \(\pi_N:E\rightarrow N\), that \(\pi_1(N)\) is generated by the curve
\begin{align*}
  \gamma_0: S^1=\lambda(F)&\rightarrow N,& z&\mapsto zx_0,
\end{align*}
where \(x_0\in N\) is any base point of \(N\).

Let \(x\in F\) be a \(T\)-fixed point.
Then, since the \(T\)-action on \(M\) is effective, up to an automorphism of \(T\), the \(T\)-representation on the tangent space at \(x\) is given by the standard representation on \(\C^n\).
Therefore \(T\) decomposes as \(T\cong (S^1)^n\), where each \(S^1\)-factor acts non-trivially on exactly one factor of \(T_xM\cong \C^n\).
It acts on this factor by complex multiplication.
Since \(\lambda(F)\) acts trivially on \(T_xF\subset T_xM\), \(\lambda(F)\) is equal to one of these \(S^1\)-factors.

Let \(T'\) be the product of the other factors.
 Then the fiber of \(\pi_F\) over \(x\) is a \(T\)-orbit of type \(T/T'\).

Then there are two cases:
\begin{enumerate}
\item \(\dim \pi_N(\pi_F^{-1}(x))=0\)
\item \(\dim \pi_N(\pi_F^{-1}(x))=1\)
\end{enumerate}

In the first case \(\pi_N(\pi_F^{-1}(x))\) is a \(T\)-fixed point \(\bar{x}_1\) in \(N\).
Because \(N\) is \(T\)-invariant, it follows from an investigation of the \(T\)-representation \(T_{\bar{x}_1}M\) that \(N\) is a fixed point component of some subtorus \(T''\) of \(T\) with \(2\dim T''=\codim N\).
Therefore \(N\) is a torus manifold.
Since \(N\) is totally geodesic in \(M\) it  follows that the induced metric on \(N\) has non-negative curvature.
Moreover \(N\) is simply connected since \(\gamma_0\) is constant for \(x_0=\bar{x}_1\).
Hence, it follows from the induction hypothesis that \(N\) is locally standard and \(N/T\) is diffeomorphic after smoothing the corners to a standard disc.
Hence it follows that the \(T\)-actions on \(D(N)\) and \(D(F)\) are locally standard and
\begin{align*}
  D(N)/T&\cong N/T\times \Delta^k \cong D^n,\\
  D(F)/T&\cong F/T\times I\cong D^n.
\end{align*}
Since \(E/T\cong F/T\) is also  diffeomorphic to a disc, it follows that \(M/T\) is diffeomorphic to a standard disc.
In particular, \(\partial M/T\) is connected.

Note that by the above arguments all characteristic submanifolds of
\(M\) are simply connected and admit an invariant metric of
non-negative curvature.
Therefore from the induction hypothesis, we know that if a facet \(\tilde{F}\) of \(M/T\) contains a vertex, then all faces contained in \(\tilde{F}\) are diffeomorphic after smoothing the corners to standard discs.
In particular each such face contains a vertex.

Since \(\partial M/T\) is connected, it follows that every facet \(\tilde{F}\) of \(M/T\) contains a vertex.
Because each proper face of \(M/T\) is contained in a facet, it follows that all faces of \(M/T\) are diffeomorphic to standard discs.
By \cite[Theorem 2]{MR2283418}, we have \(H^{\text{odd}}(M;\mathbb{Z})=0\).
Hence, the lemma follows in this case.

In the second case \(\pi_N(\pi_F^{-1}(x)))\) is a one-dimensional orbit.
Moreover,  \(\pi_F^{-1}(x)\) is an orbit of type \(T/T'\).
Since the \(T\)-action on \(M\) is effective and \(T'\) is a subtorus
of \(T\) of codimension one, it follows from dimension reasons and the
slice theorem that there is an invariant neighborhood of
\(\pi_F^{-1}(x)\) which is equivariantly diffeomorphic to 
\begin{equation}
\label{eq:4}
  \lambda(F)\times \C^{n-1} \times \R,
\end{equation}
where \(\C^{n-1}\) is a faithful \(T'\)-representation and \(\R\) is a trivial representation.
Since \(E\) has an invariant collar in \(D(F)\) and \(D(N)\), the \(\R\)-factor is normal to \(E\).

Since \(\pi_N\) is a equivariant, \(\pi_N(\pi_F^{-1}(x)))\)
is an orbit of type \(T/(H_0\times T')\), where \(H_0\) is a finite
subgroup of \(\lambda(F)\).

By an argument similar to the argument given above for
\(\pi_F^{-1}(x)\), \(\pi_N(\pi_F^{-1}(x)))\) has an invariant neighborhood
 in \(M\) which is diffeomorphic to
\begin{equation}
\label{eq:2}  \lambda(F)\times_{H_0} \C^{n-1} \times \R,
\end{equation}
where \(T'\) acts effectively on \(\C^{n-1}\) and the \(H_0\)-action on \(\C^{n-1}\times \R\) commutes with the \(T'\)-action.
Moreover, the factor \(\R\) is normal to \(N\) because the
\(\R\)-factor in (\ref{eq:4}) is normal to \(E\) and \(\pi_N\) is an
equivariant submersion.

The restriction of the tangent bundle of \(N\) to the orbit \(\pi_N(\pi_F^{-1}(x))\cong \lambda(F)/H_0\) is an invariant subbundle of the restriction of the tangent bundle of \(M\) to this orbit.
The latter is isomorphic to \(\lambda(F)\times_{H_0} \C^{n-1} \times \R\).

Because \(T'\) has dimension \(n-1\) and acts effectively on
\(\C^{n-1}\), the invariant subvector bundles of this bundle are all of the form
\begin{equation*}
  \lambda(F)\times_{H_0} \C^k \times \R^l,
\end{equation*}
with \(0\leq k\leq n-1\) and \(l=0,1\).
Since the \(\R\)-factor is normal to \(N\) and \(M\) has even dimension, it follows that \(N\) has odd dimension.

{\bf Claim:} \(\lambda(F)\) acts freely on \(N\).

Assume that there is an \(H\subset \lambda(F)\), \(H\neq \{1\}\), such that \(H\) has a fixed point \(x_2\in N\).
We may assume that \(H\) has order equal to a prime \(p\).
Then \(H\) acts freely on the fiber of \(\pi_N\) over \(x_2\).
This fiber is diffeomorphic to \(S^{2k}\).
Since \(2=\chi(S^{2k})\equiv \chi((S^{2k})^H)\mod p\), it follows that \(p=2\).
In this case the restriction of \(E\) to the orbit \(\lambda(F)x_2\) is a non-orientable sphere bundle.
Hence \(N\) is not orientable.
Therefore \(\pi_1(N)\) has even order.

Let
\begin{align*}
  \gamma_1: [0,\frac{1}{2}]&\rightarrow N,& y&\mapsto \exp(i2\pi y)x_2.
\end{align*}
Then \(\gamma_0\) is homotopic to \(2\gamma_1\).
Since \(\pi_1(N)\) is cyclic and generated by \(\gamma_0\), it follows that
\begin{align*}
  [\gamma_0]=2[\gamma_1]=2k[\gamma_0],
\end{align*}
for some \(k\in \mathbb{Z}\). 
Hence, \(0=(2k-1)[\gamma_0]\), which implies that \(\pi_1(N)\) is of odd order.
This gives a contradiction.
Therefore \(\lambda(F)\) acts freely on \(N\).
In particular \(H_0\) is trivial.

Now it follows from (\ref{eq:2}), that \(N/\lambda(F)\) is a torus manifold with \(\pi_1(N/\lambda(F))=0\).
Hence, \(N\) is orientable, because the stable tangent bundle of \(N\) is isomorphic to the pullback of the stable tangent bundle of \(N/\lambda(F)\).
Moreover, by (\ref{eq:2}), \(N\) is a codimension-one submanifold of a fixed point component \(N'\) of a subtorus \(T''\subset T\) with \(2\dim T''=\codim N'\).
The normal bundle of \(N'\) in \(M\) splits as a sum of complex line bundles.
Therefore \(N'\) is orientable and the normal bundle of \(N\) in \(N'\) is trivial.
Hence, the structure group of the normal bundle of \(N\) in \(M\) (and also of \(E\rightarrow N\)) is given by \(T''\).

Let \(T'''\) be a complimentary subtorus of \(T\) to \(T''\) with \(T'''\supset \lambda(F)\).
The \(T\)-action on \(E\) can be described as follows.
\(T''\) acts linearly on the sphere \(S^{2k}\) with \(2k=\codim N -1\).
Let \(P\) be the principal \(T''\)-bundle associated to \(E\rightarrow N\).
Then we have
\begin{equation*}
  E\cong P\times_{T''}S^{2k}.
\end{equation*}
The \(T'''\)-action on \(N\) lifts to an action on \(P\).
Together with the \(T''\)-action on \(S^{2k}\) this action induces the \(T\)-action on \(E\).

Let \(H\subset T'''/\lambda(F)\) be the isotropy group of some point \(y\in N/\lambda(F)\).
Then \(H\) acts on the fiber of \(E/\lambda(F)\) over \(y\) via a homomorphism \(\phi:H\rightarrow T''\).
This \(\phi\) depends only on the component of \((N/\lambda(F))^H\) which contains \(y\).
Since \(H'=\graph \phi^{-1}\subset T''\times T'''/\lambda(F)\) acts trivially on the fiber of \(E/\lambda(F)\rightarrow N/\lambda(F)\) over \(y\), it follows that
\begin{equation*}
  \codim (N/\lambda(F))^H=\codim (E/\lambda(F))^{H'}
\end{equation*}
and that \(H'\) is the isotropy group of generic points in the fiber over \(y\).
Since \(E/\lambda(F)\) is equivariantly diffeomorphic to \(F\) and \(F\) is locally standard by the induction hypothesis, it follows that \(H\) is a torus and \(2\dim H=\codim (N/\lambda(F))^H \).
Therefore \(N/\lambda(F)\) is locally standard.
Hence it follows that \(M\) is locally standard in a neighborhood of \(N\).
Since \(M\) is also locally standard in a neighborhood of \(F\), it follows that \(M\) is locally standard everywhere.

Now we have the following sequence of diffeomorphisms
\begin{equation*}
  \begin{split}
  D^{n-1}\cong F/T\cong E/T=(E/T'')/T'''=(P\times_{T''}(\Sigma^k))/T'''\\=N/T'''\times \Sigma^k\cong N/T'''\times D^k.    
  \end{split}
\end{equation*}
Hence there is a diffeomorphism \(D^n\cong N/T'''\times D^{k+1} \cong N/T'''\times \Sigma^{k+1} \cong D(N)/T\).
Now the statement follows as in the first case.
\end{proof}

For the proof of Theorem~\ref{sec:proof-theor-refs-1} we need some more preparation.

\begin{lemma}
\label{sec:non-negat-curv}
  Let \(Q\) be a nice manifold with corners such that all faces of \(Q\) are diffeomorphic (after smoothing the corners) to standard discs.
  Then the diffeomorphism type of \(Q\) is uniquely determined by \(\mathcal{P}(Q)\).
\end{lemma}
\begin{proof}
  This follows directly from results of Davis \cite[Theorem 4.2]{davis13:_when_coxet}.
\end{proof}

In analogy to line shellings for polytopes we define shellings for nice manifolds with corners.

\begin{definition}
  Let \(Q\) be a nice manifold with corners such that all faces of \(Q\) are contractible.
  An ordering \(F_1,\dots,F_s\) of the facets of \(Q\) is called a shelling if
  \begin{enumerate}
  \item \(F_1\) has a shelling.
  \item For \(1< j\leq s\), \(F_j\cap \bigcup_{i=1}^{j-1} F_i\) is the beginning of a shelling of \(F_j\), i.e.
    \begin{equation*}
      F_j\cap \bigcup_{i=1}^{j-1} F_i = G_1\cup\dots\cup G_r
    \end{equation*}
    for some shelling \(G_1,\dots,G_r,\dots,G_t\) of \(F_j\).
  \item If \(j<s\), then \(\bigcup_{i=1}^j F_i\) is contractible.
  \end{enumerate}

\(Q\) is called shellable if it has a shelling.
\end{definition}

\begin{example}
  \(\Delta^n\) and \(\Sigma^n\) are shellable and any ordering of their facets is a shelling.
  This follows by induction on the dimension \(n\) because the intersection of any facet \(F_j\) with a facet \(F_i\) with \(i<j\) is a facet of \(F_j\).
\end{example}

\begin{lemma}
  Let \(Q_1\) and \(Q_2\) be two nice manifolds with corners such that all faces of \(Q_1\) and \(Q_2\) are contractible.
  If \(F_1,\dots,F_s\) and \(G_1,\dots, G_r\) are shellings of \(Q_1\) and \(Q_2\), respectively, then
  \begin{equation*}
    F_1\times Q_2,\dots,F_{s-1}\times Q_2,Q_1\times G_1,\dots,Q_1\times G_r, F_s\times Q_2
  \end{equation*}
is a shelling of \(Q_1\times Q_2\).
\end{lemma}
\begin{proof}
  We prove this lemma by induction on the dimension of \(Q_1\times Q_2\).
  For \(\dim Q_1\times Q_2=0\) there is nothing to show. Therefore assume that \(\dim Q_1\times Q_2>0\) and the lemma is proved for all products \(\tilde{Q}_1\times\tilde{Q}_2\) of dimension less than \(\dim Q_1\times Q_2\).
  \begin{enumerate}
  \item It follows from the induction hypothesis that \(F_1\times Q_2\) has a shelling.
  \item For \(j\leq s-1\), we have
    \begin{equation}
      \label{eq:shell1}
      (F_j\times Q_2)\cap \bigcup_{i=1}^{j-1}(F_i\times Q_2)= (F_j\cap \bigcup_{i=1}^{j-1}F_i)\times Q_2.
    \end{equation}
    Because \(\bigcup_{i=1}^jF_i\) is contractible, \(F_j\cap \bigcup_{i=1}^{j-1}F_i\neq \partial F_j\) is the beginning of a shelling of \(F_j\).
    By the induction hypothesis it follows that (\ref{eq:shell1}) is the beginning of a shelling of \(F_j\times Q_2\).

    For \(1\leq j\leq r\) we have
 \begin{equation}
      \label{eq:shell2}
      (Q_1\times G_j)\cap \left(\bigcup_{i=1}^{s-1}(F_i\times Q_2)\cup \bigcup_{i=1}^{j-1}(Q_1\times G_i)\right) = \bigcup_{i=1}^{s-1}F_i\times G_j\cup \left(Q_1\times (\bigcup_{i=1}^{j-1}G_j\cap G_i)\right).
    \end{equation}
    Therefore it follows from the induction hypothesis that (\ref{eq:shell2}) is the beginning of a shelling of \(Q_1\times G_j\).

    The intersection of \(F_s\times Q_2\) with
 \[\bigcup_{i=1}^{s-1}(F_i\times Q_2)\cup \bigcup_{i=1}^r(Q_1\times G_i)\]
 is the whole boundary of \(F_s\times Q_2\).
    Since by assumption there are shellings for \(F_s\) and \(Q_2\), it follows from the induction hypothesis that it is a beginning of a shelling for \(F_s\times Q_2\).

  \item The verification that \(\bigcup_{i=1}^j F_i\times Q_2 \cup \bigcup_{i=1}^{j'} Q_1\times G_i\) is contractible for \(j\leq s-1\) and \(j'\leq r\) is left to the reader.
  \end{enumerate}
\end{proof}

It follows from the above lemma that \(\prod_i \Sigma^{n_i}\times \prod_i \Delta^{n_i}\) is shellable.
Now we can prove the following lemma in the same way as Theorem 5.6 in \cite{MR3030690}.

\begin{lemma}
\label{sec:proof-theor-refs-3}
Let \(M\) be a locally standard torus manifold over a shellable nice manifold with corners \(Q\).
Then \(M\) is determined up to equivariant diffeomorphism by the characteristic function \(\lambda_M\).
\end{lemma}

Now we can prove Theorem~\ref{sec:proof-theor-refs-1}.

\begin{proof}[{Proof of Theorem~\ref{sec:proof-theor-refs-1}}]
  As in the proof of Theorem~\ref{sec:torus-manifolds-with}, it follows that \((\mathcal{P}(M/T),\lambda)\) is isomorphic to \((\mathcal{P}(M'/T),\lambda')\), where \(M'\) is the quotient of a free linear torus action on a product of spheres.
  This quotient admits an invariant metric of non-negative curvature and is simply connected.
  Therefore, by Lemmas \ref{sec:proof-theor-refs-2}, \ref{sec:non-negat-curv} and \ref{sec:proof-theor-refs-3}, \(M\) and \(M'\) are equivariantly diffeomorphic.
\end{proof}

\section{Non-simply connected non-negatively curved torus manifolds}
\label{sec:non-simply-con}

Now we discuss some results for non-simply connected non-negatively curved torus manifolds.

\begin{lemma}
  \label{sec:non-negat-curv-1}
  Let \(M\) be a \(2n\)-dimensional torus manifold with an invariant metric of non-negative sectional curvature.
  Then we have \(|\pi_1(M)|=2^k\) for some \(0\leq k \leq n-1\).
\end{lemma}
\begin{proof}
  We prove this lemma by induction on the dimension \(2n\) of \(M\).
  For \(n=1\) the only \(2n\)-dimensional torus manifold is \(S^2\).
  Therefore the lemma is true in this case.

  Now assume that the lemma is true for all torus manifolds of dimension less than \(2n\).
  Let \(M\) be a torus manifold of dimension \(2n\) with a metric of non-negative curvature and \(F\) a characteristic submanifold of \(M\).
Then by Theorem~\ref{sec:proof-theor-refs} we have a decomposition of \(M\) as a union of two disk bundles:
\begin{equation*}
  M=D(N)\cup_ED(F).
\end{equation*}

Moreover, \(F\) is a torus manifold of dimension \(2(n-1)\) which admits an invariant metric of non-negative curvature.
Therefore, by the induction hypothesis, we have \(|\pi_1(F)|=2^k\) with \(0\leq k\leq n-2\).

At first assume that \(\codim N\geq 3\).
Then it follows from the exact homotopy sequence for the fibration \(E \rightarrow N\) that \(\pi_1(E)\rightarrow \pi_1(N)\) is an isomorphism.
Hence, it follows from Seifert--van Kampen's theorem that \(\pi_1(M)=\pi_1(F)\).
So the claim follows in this case.

Next assume that \(\codim N=2\).
Let \(x\in F\) be a \(T\)-fixed point.
Then with an argument similar to that in the proof of Lemma \ref{sec:proof-theor-refs-2}, one sees that \(\pi_N(\pi_F^{-1}(x))=\{y\}\) is a single point.
Here \(\pi_N:E\rightarrow N\) and \(\pi_F:E\rightarrow F\) denote the bundle projections.

Therefore \(y\in N\) is a \(T\)-fixed point.
Hence, \(N\) is a characteristic submanifold of \(M\).
Denote by \(\lambda(N)\subset T\) the circle subgroup of \(T\) which fixes \(N\).
Then we have an exact sequence
\begin{equation*}
  \pi_1(\lambda(N))\rightarrow \pi_1(E)\rightarrow \pi_1(N)\rightarrow 1.
\end{equation*}
Here the first map is induced by the inclusion of an \(\lambda(N)\)-orbit.
Now it follows from Seifert--van Kampen's theorem, that \(\pi_1(M)=\pi_1(F)/\langle \pi_1(\lambda(N))\rangle\).
Here \(\langle \pi_1(\lambda(N))\rangle\) denotes the normal subgroup of \(\pi_1(F)\) which is generated by the image of the map \(\pi_1(\lambda(N))\rightarrow \pi_1(F)\) induced by the inclusion of a \(\lambda(N)\)-orbit.
Since there are \(T\)-fixed points in \(F\), the \(\lambda(N)\)-orbits in \(F\) are null-homotopic.
Therefore it follows that \(\pi_1(M)=\pi_1(F)\).
Hence the claim follows in this case.

Now assume that \(\codim N=1\).
Then the map \(E\rightarrow N\) is a two-fold covering.
Therefore we have an exact sequence
\begin{equation}
\label{eq:3}
  1\rightarrow \pi_1(E)\rightarrow \pi_1(N)\rightarrow \mathbb{Z}_2\rightarrow 1.
\end{equation}
In particular, \(\pi_1(E)\) is a normal subgroup of \(\pi_1(N)\).

Since \(\codim F=2\), we get the following exact sequence from the exact homotopy sequence for the fibration \(E\rightarrow F\)
\begin{equation*}
  \pi_1(\lambda(F))\rightarrow \pi_1(E)\rightarrow \pi_1(F)\rightarrow 1.
\end{equation*}
Therefore it follows from Seifert--van Kampen's theorem that 
\begin{equation*}
  \pi_1(M)=\pi_1(N)/\langle\pi_1(\lambda(F))\rangle.
\end{equation*}
Here \(\langle \pi_1(\lambda(F))\rangle\) denotes the normal subgroup of \(\pi_1(N)\) which is generated by the image of the inclusion \(\pi_1(\lambda(F))\rightarrow \pi_1(E)\rightarrow \pi_1(N)\).

Since \(\pi_1(E)\subset \pi_1(N)\) is normal, we have \(\langle \pi_1(\lambda(F))\rangle \subset \pi_1(E)\).
Therefore from (\ref{eq:3}) we get the following exact sequence
\begin{equation*}
  1\rightarrow \pi_1(E)/\langle \pi_1(\lambda(F))\rangle\rightarrow \pi_1(M)\rightarrow \mathbb{Z}_2\rightarrow 1.
\end{equation*}
Since there is a surjection \(\pi_1(F)=\pi_1(E)/\pi_1(\lambda(F))\rightarrow \pi_1(E)/\langle \pi_1(\lambda(F))\rangle\), the claim now follows.
\end{proof}

As a corollary to Lemma~\ref{sec:non-negat-curv-1} we get:

\begin{cor}
\label{sec:non-simply-connected}
  Let \(M\) be a \(2n\)-dimensional torus manifold which admits an invariant metric of non-negative sectional curvature.
  Then the universal covering \(\tilde{M}\) of \(M\) is a simply connected torus manifold which admits an invariant metric of non-negative curvature.
  Moreover, the action of the torus on \(\tilde{M}\) commutes with the action of the deck transformation group.
\end{cor}
\begin{proof}
  By Lemma~\ref{sec:non-negat-curv-1}, \(\tilde{M}\) is a closed manifold.
  Since there are \(T\)-fixed points in \(M\), the principal orbits of the \(T\)-action on \(M\) are null-homotopic in \(M\).
  Hence it follows that the \(T\)-action lifts to an action on \(\tilde{M}\).

  This action on \(\tilde{M}\) has a fixed point and normalizes the deck transformation group \(G\).
  Since \(T\) is connected and \(G\) discrete it follows that the \(T\)- and \(G\)-actions on \(\tilde{M}\) commute.

  The metric on \(M\) lifts to an metric on \(\tilde{M}\) which clearly has non-negative sectional curvature and is invariant under the lifted torus action.
Hence the claim follows.
\end{proof}

Now we can determine the isomorphism type of the fundamental group of a non-simply connected non-negatively curved torus manifold.

\begin{theorem}
  Let \(M\) be a non-negatively curved torus manifold of dimension \(2n\). 
  Then there is a \(0\leq k\leq n-1\), such that \(\pi_1(M)=\mathbb{Z}_2^k\).
\end{theorem}
\begin{proof}
  By Corollary~\ref{sec:non-simply-connected}, the universal covering \(\tilde{M}\)  of \(M\) is a torus manifold.
  Moreover, the action of \(G=\pi_1(M)\) on \(\tilde{M}\) commutes with the action of the torus \(T\) on \(\tilde{M}\).

  Therefore it induces a \(G\)-action on \(\mathcal{P}(\tilde{M}/T)\).
  Moreover, for any \(g\in G\) and all faces \(F\) of \(\tilde{M}/T\), we have
  \begin{equation*}
    \lambda(gF)=\lambda(F).
  \end{equation*}
  Hence, the intersection of \(gF\) and \(F\) is empty if \(gF\neq F\).

  Since \(\mathcal{P}(\tilde{M}/T)\cong \mathcal{P}(\prod_{i< r}
  \Sigma^{n_i}\times \prod_{i\geq r}\Delta^{n_i})\) and all facets of \(\Sigma^{n_i}\) and \(\Delta^{n_i}\) have non-trivial intersection if \(n_i\geq 2\), it follows that \(gF=F\) for all facets of \(\tilde{M}/T\) belonging to a factor of dimension at least two.
Moreover, the facets which belong to the other factors are mapped to facets that belong to the same factor.

Since \(G\) acts freely on \(\tilde{M}^T\), the \(G\)-action on \(\mathcal{P}(\tilde{M}/T)\) is effective.
Therefore \(G\) might be identified with a subgroup of \(\Aut(\mathcal{P}(\tilde{M}/T))\).
This subgroup is contained in the subgroup \(H\) of \(\Aut(\mathcal{P}(\tilde{M}/T))\) which contains all automorphisms which leave all facets belonging to factors of dimension greater or equal to two invariant and maps facets belonging to a factor of dimension one to facets belonging to the same factor.

We will show that \(H\) is isomorphic to \(\mathbb{Z}_2^{l+r-1}\)
where \(l\) is the number of \(i\geq r\) with \(n_i=1\).
Using Lemma~\ref{sec:non-negat-curv-1} one sees that this implies the theorem.

We define a homomorphism \(\psi:H\rightarrow\mathbb{Z}_2^{l+r-1}\) as follows.

At first assume that the factor \(\Sigma^{n_i}\) has dimension at least two.
Then we set \(\psi(h)_i=0\) if and only if \(h\) leaves all components of \(\bigcap_j F_j\) invariant, where the intersection is taken over all facets \(F_j\) of \(\tilde{M}/T\) belonging to the factor \(\Sigma^{n_i}\).
Note that this intersection has two components.

Now assume that the factor \(\Gamma_i\) has dimension one.
Then we set \(\psi(h)_i=0\) if and only if \(h\) leaves the two facets belonging to \(\Gamma_i\) invariant.

The homomorphism \(\psi\) has an obvious inverse \(\phi:\mathbb{Z}_2^{l+r-1}\rightarrow H\). It is defined as follows.
For \(a\in \mathbb{Z}_2^{l+r-1}\), \(\phi(a)\) leaves all facets of \(\tilde{M}/T\) which do not belong to a factor of dimension two invariant.

Now assume that the factor \(\Sigma^{n_i}\) has dimension at least two.
Then \(\phi(a)\) interchanges the two components of \(\bigcap_j F_j\), where the intersection is taken over all facets belonging to \(\Sigma^{n_i}\), if and only if \(a_i\neq 0\).
Otherwise it leaves these components invariant.

Now assume that the factor \(\Gamma_i\) has dimension one.
Then \(\phi(a)\) interchanges the two facets belonging to \(\Gamma_i\) if and only if \(a_i\neq 0\).
Otherwise it leaves these facets invariant.
It is easy to check that \(\phi(a)\) defined as above extends to an
automorphism of \(\mathcal{P}(\tilde{M}/T)\). 

So we see that \(\psi\) is an isomorphism and the theorem is proved.
\end{proof}

We give an example to show that the bound on the order of the fundamental group given in the above theorem is sharp.

\begin{example}
  Let \(\tilde{M}=\prod_{i=1}^n S^2\) with the torus action induced by rotating each factor.
  The product metric of the standard metrics on each factor is invariant under the action of the torus and has non-negative curvature. 
  On each factor there is an isometric involution \(\iota_1\) given by the antipodal map.

  We define an action of \(\mathbb{Z}^{n-1}_2\) on \(\tilde{M}\) as follows.
  Let \(e_1,\dots,e_{n-1}\) a generating set of \(\mathbb{Z}^{n-1}_2\).
  Then each \(e_i\) acts on the \(i\)-th factor and \((i+1)\)-st factor of \(\tilde{M}\) by \(\iota_1\) and trivially on the other factors.
  This defines a free orientation preserving action of \(\mathbb{Z}^{n-1}_2\) which commutes with the torus action.

  Therefore \(\tilde{M}/\mathbb{Z}^{n-1}_2\) is a torus manifold with an invariant metric of non-negative sectional curvature and fundamental group \(\mathbb{Z}^{n-1}_2\).
\end{example}

\bibliography{non-negative-cur}{}
\bibliographystyle{amsplain}
\end{document}